\documentclass[11pt]{article}
\usepackage{latexsym,amsfonts,amssymb,amsmath,euscript,amsthm}
\usepackage[applemac]{inputenc}  
\usepackage[english]{babel}  
\usepackage{graphicx} 
\usepackage{subfigure}
\usepackage{hyperref}         
\usepackage{color,fancybox}
\usepackage{float}

  \usepackage{paralist}
  
\usepackage{upgreek}
  \usepackage{epsfig} 
\usepackage{graphicx}  \usepackage{epstopdf}
\usepackage{bm}
\hypersetup{urlcolor=blue, citecolor=red}

\newcommand{\eps}{\varepsilon}

\newcommand{\eto}{\stackrel{\eps\to 0}{\longrightarrow}}

\newcommand{\weto}{\stackrel{\eps\to 0}{\rightharpoonup}}
\newcommand{\R} {\mathbb{R}}

\newcommand{\Z} {\mathbb{Z}}

\def\eto{\buildrel \epsilon\to 0\over\longrightarrow }

\title{$C^{1,\theta}$-Estimates on the distance of Inertial Manifolds\footnote{
This research has been partially supported by grants MTM2016-75465,  MTM2012-31298, ICMAT Severo Ochoa project SEV-2015-0554 (MINECO), Spain and Grupo de Investigaci\'on CADEDIF, UCM.}   }
\author{Jos\'{e} M. Arrieta\footnote{Departamento de Matem\'atica Aplicada, Universidad Complutense de Madrid, 28040 Madrid and Instituto de Ciencias Matem\'aticas
CSIC-UAM-UC3M-UCM, Spain.   e-mail: arrieta@mat.ucm.es}
and Esperanza Santamar\'ia\footnote{Universidad a Distancia de Madrid, 28400 Collado Villalba, Madrid.  email: esperanza.santamaria@udima.es} 
}

\date{ }

\begin{document}

\maketitle
%
%
%
%

{\footnotesize 
\par\noindent {\bf Abstract:}
In this paper we obtain $C^{1,\theta}$-estimates on the distance of inertial manifolds for dynamical systems generated by 
evolutionary parabolic type equations. We consider the situation where the systems are defined in different phase spaces and we estimate the distance in terms of the distance of the resolvent operators of the corresponding elliptic operators and the distance of the nonlinearities of the equations.
 \vskip 0.5\baselineskip

%

\vspace{11pt}

\vspace{6pt}


}

\numberwithin{equation}{section}
\newtheorem{teo}{Theorem}[section]
\newtheorem{lem}[teo]{Lemma}
\newtheorem{cor}[teo]{Corolary}
\newtheorem{prop}[teo]{Proposition}
\newtheorem{defi}[teo]{Definition}
\newtheorem{re}[teo]{Remark}

%

\section{Introduction}
We continue in this work the analysis started in \cite{Arrieta-Santamaria-DCDS} on the estimates on the distance of inertial manifolds.  Actually, in \cite{Arrieta-Santamaria-DCDS} we considered a family of abstract evolution equations of parabolic type, that may be posed in different phase spaces (see equation (\ref{problemaperturbado}) below)  and we impose very general conditions (see (H1) and (H2) below)  guaranteing that each problem has an inertial manifold and more important,  we were able to obtain estimates in the norm of the supremum on the convergence of the inertial manifolds. These estimates are expressed  in terms of the distance of the resolvent operators and in terms of the distance of the nonlinear terms.  These results are the starting point of the present paper and are briefly described in Section \ref{setting} (see Theorem \ref{distaciavariedadesinerciales})

One of the main applications of invariant manifolds is that they allow us to describe the dynamics (locally or globally) of an infinite dimensional system with only a finite number of parameters (the dimension of the manifold). This drastic reduction of dimensionality permits in many instances to analyze in detail the dynamics of the equation and study perturbations problem. But for these questions, some extra differentiability on the manifold and some estimates on the convergence on stronger norms like $C^1$ or $C^{1,\theta}$ is desirable, see \cite{Hale&Raugel3,Arrieta-Santamaria-2}. Actually, the estimates from this paper and from \cite{Arrieta-Santamaria-DCDS} are key estimates to obtain good rates on the convergence of attractors of reaction diffusion equations in thin domains, problem which is addressed in \cite{Arrieta-Santamaria-2}. 

This is actually the main purpose of this work.  Under the very general setting from \cite{Arrieta-Santamaria-DCDS} but impossing some extra differentiability and convergence properties on the nonlinear terms (see hipothesis (H2') below) we obtain that the inertial manifolds are uniformly $C^{1,\theta}$ smooth and obtain estimates on the convergence of the manifolds in this $C^{1,\theta}$ norm.


Let us mention that the theory of invariant and inertial manifolds is a well established theory.  We refer to \cite{Bates-Lu-Zeng1998, Sell&You} for general references on the theory of Inertial manifolds. See also \cite{JamesRobinson} for an accessible introduction to the theory. These inertial manifolds are smooth, see \cite{ChowLuSell}. We also refer to \cite{Henry1, Hale, B&V2,Sell&You,LibroAlexandre,Cholewa}
for general references on dynamics of evolutionary equations.

We describe now the contents of the paper. 

In Section \ref{setting} we introduce the notation, review the main hypotheses (specially (H1) and (H2)) and results from \cite{Arrieta-Santamaria-DCDS}. We describe in detail the new hypothesis (H2') and state the main result  of the paper, Proposition \ref{FixedPoint-E^1Theta} and Theorem \ref{convergence-C^1-theo}.

In Section  \ref{smoothness} we analyze the $C^{1,\theta}$ smoothness of the inertial manifold, proving Proposition \ref{FixedPoint-E^1Theta}. The analysis is based in previous results from \cite{ChowLuSell}.  

In Section \ref{convergence} we obtain the estimates on the distance of the inertial manifold in the $C^{1,\theta}$ norm, proving Theorem \ref{convergence-C^1-theo}. 


\section{Setting of the problem and main results}
\label{setting}

In this section we consider the setting of the problem, following \cite{Arrieta-Santamaria-DCDS}. We refer to this paper for more details about the setting.

%
%
%

Hence, consider the family of problems,
\begin{equation}\label{problemalimite}
(P_0)\left\{
\begin{array}{r l }
u^0_t+A_0u^0&=F_0^\eps(u^0),\\
u^0(0)\in X^\alpha_0,
\end{array}
\right.
\end{equation}
and 
\begin{equation}\label{problemaperturbado}
(P_\varepsilon)\left\{
\begin{array}{r l }
u^\varepsilon_t+A_\varepsilon u^\varepsilon&=F_\varepsilon(u^\varepsilon),\qquad 0<\varepsilon\leq \eps_0\\
u^\varepsilon(0)\in X^\alpha_\varepsilon,
\end{array}
\right.
\end{equation}
where we assume, that $A_\varepsilon$ is self-adjoint positive linear operator on a separable real Hilbert space $X_\varepsilon$,  that is 
$A_\varepsilon: D(A_\varepsilon)=X^1_\varepsilon\subset X_\varepsilon\rightarrow X_\varepsilon,$
and  $F_\eps:X_\eps^\alpha\to X_\eps$, $F_0^\eps:X_0^\alpha\to X_0$ are nonlinearities guaranteeing  global existence of solutions of \eqref{problemaperturbado}, for each $0\leq \eps\leq \eps_0$ and for some $0\leq \alpha<1$.   Observe that for problem \eqref{problemalimite} we even assume that the nonlinearity depends on $\eps$ also. 


%

As in \cite{Arrieta-Santamaria-DCDS}, we assume the existence of linear continuous operators, $E$ and $M$, such that, $E: X_0\rightarrow X_\varepsilon$, $M: X_\varepsilon\rightarrow X_0$ and $E_{\mid_{X^\alpha_0}}: X_0^\alpha\rightarrow X_\varepsilon^\alpha$ and $M_{\mid_{X_\varepsilon^\alpha}}: X_\varepsilon^\alpha\rightarrow X_0^\alpha$, satisfying,
\begin{equation}\label{cotaextensionproyeccion}
\|E\|_{\mathcal{L}(X_0, X_\varepsilon)}, \|M\|_{\mathcal{L}(X_\varepsilon, X_0)}\leq \kappa ,\qquad \|E\|_{\mathcal{L}(X^\alpha_0, X^\alpha_\varepsilon)}, \|M\|_{\mathcal{L}(X^\alpha_\varepsilon, X^\alpha_0)}\leq \kappa.
\end{equation}
for some constant $\kappa\geq 1$.   
We also assume these operators satisfy the following properties,
\begin{equation}\label{propiedadesextensionproyeccion}
M\circ E= I,\qquad \|Eu_0\|_{X_\varepsilon}\rightarrow \|u_0\|_{X_0}\quad\textrm{for}\quad u_0\in X_0.
\end{equation}

%
The family of operators $A_\varepsilon$, for $0\leq \eps\leq \eps_0$,  have compact resolvent. This,  together with the fact that the operators are selfadjoint,  implies that its spectrum is discrete real and consists only of eigenvalues, each one with finite multiplicity. Moreover, the fact that $A_\varepsilon$, $0\leq \varepsilon\leq \eps_0$, is positive implies that its spectrum is positive. So, we denote by $\sigma(A_\varepsilon)$,  the spectrum of the operator $A_\varepsilon$,  with, 
$$\sigma(A_\varepsilon)=\{\lambda_n^\varepsilon\}_{n=1}^\infty,\qquad\textrm{ and}\quad 0<c\leq\lambda_1^\varepsilon\leq\lambda_2^\varepsilon\leq...\leq\lambda_n^\varepsilon\leq...$$
and we also denote by $\{\varphi_i^\varepsilon\}_{i=1}^\infty$ an associated orthonormal family of eigenfunctions. Observe that the requirement of the operators $A_\eps$ being positive can be relaxed to requiring that they are all bounded from below uniformly in the parameter $\epsilon$. We can always consider the modified operators $A_\eps+cI$  with $c$ a large enough constant to make the modified operators positive.  The nonlinear equations \eqref{problemaperturbado} would have to be rewritten accordingly.  

\par\bigskip 

With respect to the relation between both operators, $A_0$ and $A_\eps$  and following \cite{Arrieta-Santamaria-DCDS}, we will assume the following hypothesis

\vspace{0.5cm}
{\sl \paragraph{\textbf{(H1).}}  With $\alpha$ the exponent from problems (\ref{problemaperturbado}), we have
\begin{equation}\label{H1equation}
\|A_\varepsilon^{-1}- EA_0^{-1}M\|_{\mathcal{L}(X_\varepsilon, X_\varepsilon^\alpha)}\to 0\quad \hbox{ as } \eps\to 0.
\end{equation}
}
\par\bigskip 


Let us define $\tau(\eps)$ as an increasing function of $\eps$ such that 
\begin{equation}\label{definition-tau}
\|A_\varepsilon^{-1}E- EA_0^{-1}\|_{\mathcal{L}(X_0, X_\varepsilon^\alpha)}\leq \tau(\eps).
\end{equation}

\par\bigskip

We also recall hypothesis {\bf(H2)} from \cite{Arrieta-Santamaria-DCDS}, regarding the nonlinearities $F_0$ and $F_\eps$,   \par\bigskip
{\sl 
\paragraph{\textbf{(H2).}} We assume that the nonlinear terms  $F_\varepsilon: X^\alpha_\varepsilon\rightarrow X_\varepsilon$  and   $F_0^\varepsilon: X^\alpha_0\rightarrow X_0$ for $0< \eps\leq \eps_0$, satisfy:

\begin{enumerate}
\item[(a)]  They are uniformly bounded, that is, there exists a constant $C_F>0$ independent of $\varepsilon$ such that,
$$\|F_\varepsilon\|_{L^\infty(X_\varepsilon^\alpha, X_\varepsilon)}\leq C_F, \quad \|F_0^\varepsilon\|_{L^\infty(X_0^\alpha, X_0)}\leq C_F$$
\item[(b)] They are globally Lipschitz on $X^\alpha_\varepsilon$ with a uniform Lipstichz constant $L_F$, that is,
\begin{equation}\label{LipschitzFepsilon}
\|F_\varepsilon(u)- F_\varepsilon(v)\|_{X_\varepsilon}\leq L_F\|u-v\|_{X_\varepsilon^\alpha}
\end{equation}
\begin{equation}\label{LipschitzF0}
\|F_0^\varepsilon(u)- F_0^\varepsilon(v)\|_{X_0}\leq L_F\|u-v\|_{X_0^\alpha}. 
\end{equation}

\item[(c)] They have a uniformly bounded support for $0<\varepsilon\leq \eps_0$: there exists $R>0$ such that 
$$Supp F_\varepsilon\subset D_{R}=\{u_\varepsilon\in X_\varepsilon^\alpha: \|u_\varepsilon\|_{X_\varepsilon^\alpha}\leq R\}$$
$$Supp F_0^\varepsilon\subset D_{R}=\{u_0\in X_0^\alpha: \|u_0\|_{X_0^\alpha}\leq R\}.$$

\item[(d)]  $F_\eps$ is near $F_0^\eps$ in the following sense,
\begin{equation}\label{estimacionefes}
\sup_{u_0\in X^\alpha_0}\|F_\varepsilon (Eu_0)-EF_0^\eps (u_0)\|_{X_\varepsilon}=\rho(\varepsilon),
\end{equation}
and $\rho(\varepsilon)\rightarrow 0$ as $\varepsilon\rightarrow 0$.

\end{enumerate} 
}

With {\bf (H1)} and {\bf (H2)} we were able to show in \cite{Arrieta-Santamaria-DCDS} the existence,  convergence and obtain some rate of the convergence in the norm of the supremum of inertial manifolds.  In order to explain the result and to understand the rest of this paper, we need to introduce several notation and results from \cite{Arrieta-Santamaria-DCDS}. We refer to this paper for more explanations.

%
%

Let us consider $m\in\mathbb{N}$ such that $\lambda_m^0<\lambda_{m+1}^0$ and denote by $\mathbf{P}_{\mathbf{m}}^{\bm\varepsilon}$ the canonical orthogonal projection onto the eigenfunctions, $\{\varphi^\varepsilon_i\}_{i=1}^m$, corresponding to the first $m$ eigenvalues of the operator $A_\varepsilon $, $0\leq\varepsilon\leq\varepsilon_0$ and $\mathbf{Q}^{\bm\varepsilon}_{\mathbf{m}}$ the projetion over its orthogonal complement, see \cite{Arrieta-Santamaria-DCDS}. For technical reasons, we express any element belonging to the linear subspace $\mathbf{P}_{\mathbf{m}}^{\bm\varepsilon}(X_\varepsilon)$ as a linear combination of the elements of the following basis
$$\{\mathbf{P}_{\mathbf{m}}^{\bm\varepsilon}(E\varphi^0_1), \mathbf{P}_{\mathbf{m}}^{\bm\varepsilon}(E\varphi^0_2), ...,\mathbf{P}_{\mathbf{m}}^{\bm\varepsilon}(E\varphi^0_m)\},\qquad\textrm{for}\quad 0\leq\varepsilon\leq\varepsilon_0,$$
with $\{\varphi^0_i\}_{i=1}^m$ the eigenfunctions related to the first $m$ eigenvalues of $A_0$, which constitute a basis in $\mathbf{P}_{\mathbf{m}}^{\bm\varepsilon}(X_\varepsilon)$ and 
in $\mathbf{P}_{\mathbf{m}}^{\bm\varepsilon}(X_\varepsilon^\alpha)$, see \cite{Arrieta-Santamaria-DCDS}.  We will denote by $\psi_i^\eps=\mathbf{ P_m^\eps}(E\varphi_i^0)$.

Let us denote by $j_\varepsilon$ the isomorphism from $ \mathbf{P}_{\mathbf{m}}^{\bm\varepsilon}(X_\varepsilon)=[\psi_1^\varepsilon, ..., \psi_m^\varepsilon]$ onto $\mathbb{R}^m$, that gives us the coordinates of each vector. That is,

\begin{equation}\label{definition-jeps}
\begin{array}{rl}
j_\varepsilon:\mathbf{P}_{\mathbf{m}}^{\bm\varepsilon}(X_\varepsilon)&\longrightarrow \mathbb{R}^m, \\
w_\varepsilon&\longmapsto\bar{p},
\end{array}
\end{equation}
where $w_\varepsilon=\sum^m_{i=1} p_i\psi^\varepsilon_i$ and $\bar{p}=(p_1, ..., p_m)$.

We denote by $|\cdot|$ the usual euclidean norm in $\mathbb{R}^m$, that is $|\bar{p}|=\left(\sum_{i=1}^mp_i^2\right)^{\frac{1}{2}}$, 
and by $|\cdot|_{\eps,\alpha}$ the following weighted one,
\begin{equation}\label{normaalpha}
|\bar{p}|_{\eps,\alpha}=\left(\sum_{i=1}^mp_i^2(\lambda_i^\varepsilon)^{2\alpha}\right)^{\frac{1}{2}}.
\end{equation}
\vspace{0.4cm}
We consider the spaces $(\mathbb{R}^m, |\cdot|)$ and $(\mathbb{R}^m, |\cdot|_{\eps,\alpha})$, that is, $\mathbb{R}^m$ with the norm $|\cdot|$ and $|\cdot|_{\eps,\alpha}$, respectively, and notice that for $w_0=\sum^m_{i=1} p_i\psi^0_i$ and $0\leq\alpha<1$ we have that,
\begin{equation}\label{normajepsilon}
\|w_0\|_{X^\alpha_0}=|j_0(w_0)|_{\eps,\alpha}.
\end{equation}
It is also not difficult to see that from the convergence of the eigenvalues (which is obtained from {\bf (H1)}, see \cite{Arrieta-Santamaria-DCDS}), we have that for a fixed $m$ and for all $\delta>0$ small enough there exists $\eps=\eps(\delta)>0$ such that 
\begin{equation}\label{des-normas}
(1-\delta)|\bar p|_{0,\alpha}\leq |\bar p|_{\eps,\alpha}\leq (1+\delta)|\bar p|_{0,\alpha}, \quad 0\leq \eps\leq \eps(\delta), \quad \forall \bar p\in\R^m.
\end{equation}

\bigskip

With respect to the behavior of the linear semigroup in the subspace $\mathbf{Q}^{\bm\varepsilon}_{\mathbf{m}}X_\eps^\alpha$, notice that we have the expression 
$$e^{-A_\eps t}\mathbf{Q}^{\bm\varepsilon}_{\mathbf{m}} u=e^{-A_\varepsilon \mathbf{Q}^{\bm\varepsilon}_{\mathbf{m}} t}u=\sum_{i=m+1}^\infty e^{-\lambda_i^\varepsilon t}(u, \varphi_i^\varepsilon)\varphi_i^\varepsilon.$$
Hence, using the expression of $e^{-A_\eps t \mathbf{Q}^{\bm\varepsilon}_{\mathbf{m}} t}$ from above and following a similar proof as Lemma 3.1 from \cite{Arrieta-Santamaria-DCDS}, we get
$$\|e^{-A_\varepsilon \mathbf{Q}^{\bm\varepsilon}_{\mathbf{m}} t}\|_{\mathcal{L}(X_\varepsilon, X_\varepsilon)}\leq e^{-\lambda_{m+1}^\varepsilon t},$$ 
and,
\begin{equation}\label{semigrupoproyectado}
\|e^{-A_\varepsilon \mathbf{Q}^{\bm\varepsilon}_{\mathbf{m}} t}\|_{\mathcal{L}(X_\varepsilon, X_\varepsilon^\alpha)}\leq e^{-\lambda_{m+1}^\varepsilon t}\left(\max\{\lambda_{m+1}^\varepsilon, \frac{\alpha}{t}\}\right)^\alpha,
\end{equation}
for $t\geq 0.$

In a similar way, we have 
$$e^{-A_\eps t}\mathbf{P}^{\bm\varepsilon}_{\mathbf{m}} u=\sum_{i=1}^m e^{-\lambda_i^\varepsilon t}(u, \varphi_i^\varepsilon)\varphi_i^\varepsilon.$$
and following similar steps as above, for $t\leq 0$ we have,
\begin{equation}\label{semigrupoproyectadoP}
\|e^{-A_\varepsilon \mathbf{P}^{\bm\varepsilon}_{\mathbf{m}} t}\|_{\mathcal{L}(X_\varepsilon, X_\varepsilon)}\leq e^{-\lambda_{m}^\varepsilon t},\,\,\,\,\,\,\,\,\,\,\,\, \|e^{-A_\varepsilon \mathbf{P}^{\bm\varepsilon}_{\mathbf{m}} t}\|_{\mathcal{L}(X^\alpha_\varepsilon, X^\alpha_\varepsilon)}\leq e^{-\lambda_{m}^\varepsilon t},
\end{equation}
\begin{equation}\label{semigrupoproyectadoP-alpha} 
\|e^{-A_\varepsilon \mathbf{P}^{\bm\varepsilon}_{\mathbf{m}} t}\|_{\mathcal{L}(X_\varepsilon, X_\varepsilon^\alpha)}\leq e^{-\lambda_m^\varepsilon t}(\lambda_m^\eps)^\alpha.
\end{equation}

\par\bigskip 
We are looking for inertial manifolds for system \eqref{problemaperturbado} and \eqref{problemalimite} which will be obtained as graphs of appropriate functions.  This motivates the introduction of the sets
$\mathcal{F}_\eps(L,\rho)$ defined as 
$$\mathcal{F}_\eps(L,\rho)=\{ \Phi :\mathbb{R}^m\rightarrow\mathbf{Q}_{\mathbf{m}}^{\bm\varepsilon}(X^\alpha_\varepsilon),\quad\textrm{such that}\quad \textrm{supp } \Phi\subset B_R\quad \textrm{and}\quad$$
$$\quad \|\Phi(\bar{p}^1)-\Phi(\bar{p}^2)\|_{X^\alpha_\varepsilon}\leq L|\bar{p}^1-\bar{p}^2|_{\eps,\alpha} \quad\bar{p}^1,\bar{p}^2\in\mathbb{R}^m \}.$$ 
Then we can show the following result.
\begin{prop} (\cite{Arrieta-Santamaria-DCDS})\label{existenciavariedadinercial}
Let hypotheses {\bf (H1)} and {\bf (H2)} be satisfied. Assume also that $m\geq 1$ is such that,
\begin{equation}\label{CondicionAutovaloresFuerte0}
\lambda_{m+1}^0-\lambda_m^0\geq 3(\kappa+2)L_F\left[(\lambda_m^0)^\alpha+(\lambda_{m+1}^0)^\alpha\right],
\end{equation}
and
\begin{equation}\label{autovalorgrande0}
(\lambda_m^0)^{1-\alpha}\geq 6(\kappa +2)L_F(1-\alpha)^{-1}.
\end{equation}
Then, there exist $L<1$ and $\varepsilon_0>0$ such that for all $0<\varepsilon\leq\varepsilon_0$ there exist  inertial manifolds $\mathcal{M}_\varepsilon$ and $\mathcal{M}_0^\varepsilon$  for    (\ref{problemaperturbado}) and (\ref{problemalimite}) respectively,  given by the ``graph'' of a function $\Phi_\varepsilon\in\mathcal{F}_\eps(L,\rho)$ and $\Phi_0^\varepsilon\in\mathcal{F}_0(L,\rho)$. 
\end{prop}

\begin{re} 
%
We have written quotations in the word ``graph'' since the manifolds  $\mathcal{M}_\varepsilon$, $\mathcal{M}_0^\varepsilon$ are not properly speaking the graph of the functions $\Phi_\varepsilon$, $\Phi_0^\varepsilon$ but rather the graph of the appropriate function obtained via the isomorphism $j_\eps$ which identifies $\mathbf{P}_{\mathbf{m}}^{\bm\varepsilon}(X_\varepsilon^\alpha)$ with $\R^m$. 
That is, $\mathcal{M}_\eps=\{ j_\eps^{-1}(\bar p)+\Phi_\varepsilon(\bar p); \quad \bar p\in \R^m\}$ and 
$\mathcal{M}_0^\eps=\{ j_0^{-1}(\bar p)+\Phi_0^\varepsilon(\bar p); \quad \bar p\in \R^m\}$
\end{re}

The main result from \cite{Arrieta-Santamaria-DCDS} was the following:

\begin{teo} (\cite{Arrieta-Santamaria-DCDS})
\label{distaciavariedadesinerciales}
Let hypotheses {\bf (H1)} and {\bf (H2)} be satisfied and let $\tau(\eps)$ be defined by  \eqref{definition-tau}. 
Then, under the hypothesis of Proposition \ref{existenciavariedadinercial}, if $\Phi_\varepsilon$ are the maps  that give us the inertial manifolds for $0<\eps\leq \eps_0$ then we have,
\begin{equation}\label{distance-inertialmanifolds}
\|\Phi_\varepsilon-E\Phi_0^\eps\|_{L^\infty(\mathbb{R}^m, X^\alpha_\varepsilon)}\leq C[\tau(\varepsilon)|\log(\tau(\varepsilon))|+\rho(\varepsilon)],
\end{equation}
with $C$ a constant independent of $\varepsilon$.
\end{teo}
\par\bigskip

\begin{re}
Properly speaking,  in \cite{Arrieta-Santamaria-DCDS} the above theorem is proved only for the case for which the nonlinearity $F_0^\eps$ from \eqref{problemalimite} satisfies $F_0^\eps\equiv F_0$ for all $0<\eps<\eps_0$. But revising the proof of \cite{Arrieta-Santamaria-DCDS} we can see that exactly the same argument is valid for the most general case where the nonlinearity depends on $\eps$. 

\end{re}

To obtain stronger convergence results on the inertial manifolds, we will need to requiere stronger conditions on the nonlinearites.  These conditions are stated in the following hypothesis, 
\par\bigskip
{\sl 
\paragraph{\textbf{(H2').}}
 We assume that the nonlinear terms $F_\varepsilon$ and $F_0^\eps$, satisfy hipothesis {\bf(H2)} and they are uniformly $C^{1,\theta_F}$ functions from $ X_\varepsilon^\alpha$ to $X_\varepsilon$, and $X_0^\alpha$ to $X_0$ respectively, for some $0<\theta_F\leq 1$. That is, $F_\eps\in C^1(X_\varepsilon^\alpha, X_\varepsilon)$, $F_0^\eps\in C^1(X_0^\alpha, X_0)$ and there exists  $L>0$, independent of $\eps$, such that 
$$\|DF_\varepsilon(u)-DF_\varepsilon(u')\|_{\mathcal{L}(X_\varepsilon^\alpha, X_\varepsilon)}\leq L\|u-u'\|^{\theta_F}_{X_\varepsilon^\alpha},\qquad \forall u, u'\in X_\varepsilon^\alpha.$$
$$\|DF_0^\varepsilon(u)-DF_0^\varepsilon(u')\|_{\mathcal{L}(X_0^\alpha, X_0)}\leq L\|u-u'\|^{\theta_F}_{X_0^\alpha},\qquad \forall u, u'\in X_0^\alpha.$$

}

\par\bigskip 
We can state now the main results of this section.
\begin{prop}\label{FixedPoint-E^1Theta}
 Assume hypotheses {\bf(H1)} and {\bf(H2')} are satisfied and that the gap conditions \eqref{CondicionAutovaloresFuerte0}, \eqref{autovalorgrande0} hold. Then, for any $\theta>0$ such that $\theta\leq\theta_F$ and $\theta<\theta_0$, where
 \begin{equation}\label{theta}
\theta_0= \frac{\lambda_{m+1}^0-\lambda_m^0-4L_F(\lambda_m^0)^\alpha-2L_F(\lambda_{m+1}^0)^\alpha}{2L_F(\lambda_m^0)^\alpha+\lambda_m^0}
\end{equation}
then, the functions $\Phi_\eps$, and $\Phi_0^\eps$ for  $0< \eps\leq\eps_0$, obtained above,  which give the inertial manifolds, are $C^{1, \theta}(\mathbb{R}^m, X_\eps^\alpha)$ and $C^{1, \theta}(\mathbb{R}^m, X_0^\alpha)$.  Moreover, the $C^{1,\theta}$ norm is bounded uniformly in $\eps$, for $\eps>0$ small. 
\end{prop}

The main result we want to show in this article is the following:

\par\bigskip
\begin{teo}\label{convergence-C^1-theo}
Let hypotheses {\bf (H1)},  {\bf (H2')} and gap conditions \eqref{CondicionAutovaloresFuerte0}, 
\eqref{autovalorgrande0} be satisfied, so that Proposition \ref{FixedPoint-E^1Theta} hold, and we have inertial manifolds 
$\mathcal{M}^\eps$, $\mathcal{M}_0^\eps$ given as the graphs of the functions $\Phi_\eps$, $\Phi_0^\eps$ for $0<\eps\leq\eps_0$.   
If we denote by
\begin{equation}\label{convergenceDF}
\beta(\eps)=\sup_{u\in \mathcal{M}^\eps_0}\|DF_\varepsilon \big(Eu \big)E-EDF_0^\eps\big(u\big)\|_{\mathcal{L}(X_0^\alpha, X_\varepsilon)},
\end{equation}
%
%
%
%
%
%
%
 then, there exists $\theta^*$ with $0<\theta^*<\theta_F$ such that for all $0<\theta<\theta^*$, we obtain the following estimate
\begin{equation}\label{distance-C^1-inertialmanifolds}
\|\Phi_\varepsilon-E\Phi_0^\eps\|_{C^{1, \theta}(\mathbb{R}^m, X_\varepsilon^\alpha)}\leq \mathbf{C} \left(\left[\beta(\varepsilon)+\Big(\tau(\varepsilon)|\log(\tau(\varepsilon))|+\rho(\varepsilon)\Big)^{\theta^*}\right]\right)^{1-\frac{\theta}{\theta^*}},
 \end{equation}
where $\tau(\eps)$, $\rho(\varepsilon)$ are given by (\ref{definition-tau}), (\ref{estimacionefes}), respectively and $\mathbf{C}$ is a constant independent of $\eps$.
\end{teo}

\bigskip
 \begin{re}
 As a matter of fact, $\theta^*$ can be chosen $\theta^*<\min\{\theta_F, \theta_0, \theta_1\}$ where $\theta_F$ is from {\bf(H2')}, $\theta_0$ is defined in\eqref{theta} and $\theta_1$,
 $$\theta_1= \frac{\lambda_{m+1}^0-\lambda_m^0-4L_F(\lambda_m^0)^\alpha}{(\kappa+2)L_F(\lambda_m^0)^\alpha+\lambda_m^0+3}, $$
 see \eqref{theta1}.
 \end{re}
As usual, we denote by $C^{1, \theta}(\mathbb{R}^m, X_\varepsilon^\alpha)$  the  space of $C^1(\mathbb{R}^m, X_\varepsilon^\alpha)$ maps whose differentials are uniformly H\"{o}lder continuous with H\"{o}lder exponent $\theta$. That is, there is a constant $C$ independent of $\eps$ such that,
$$\|D\Phi_\eps(z)-D\Phi_\eps(z')\|_{\mathcal{L}(\mathbb{R}^m, X_\eps^\alpha)}\leq C|z-z'|_{\eps,\alpha}^{\theta}.$$
where the norm $|\cdot |_{\eps,\alpha}$ is given by (\ref{normaalpha}). Notice that the norm $|\cdot|_{\eps, \alpha}$ is equivalent to $|\cdot|$ uniformly in $\eps$ and $\alpha$.

The space $C^{1, \theta}(\mathbb{R}^m, X_\varepsilon^\alpha)$ is endowed with the norm $\|\cdot\|_{C^{1, \theta}(\mathbb{R}^m, X_\varepsilon^\alpha)}$ given by,
$$\|\Phi_\eps\|_{C^{1, \theta}(\mathbb{R}^m, X_\varepsilon^\alpha)}=\|\Phi_\eps\|_{C^1(\mathbb{R}^m, X_\varepsilon^\alpha)}+\sup_{z, z'\in\mathbb{R}^m}\frac{\|D\Phi_\eps(z)-D\Phi_\eps(z')\|_{\mathcal{L}(\mathbb{R}^m, X_\eps^\alpha)}}{|z-z'|_{\eps,\alpha}^{\theta}}$$

\bigskip 

To simplify notation below and unless some clarification is needed,  we will denote the norms $\|\cdot\|_{C^{1}(\mathbb{R}^m, X_\varepsilon^\alpha)}$ and  $\|\cdot\|_{C^{1, \theta}(\mathbb{R}^m, X_\varepsilon^\alpha)}$  by  $\|\cdot\|_{C^{1}}$ and  $\|\cdot\|_{C^{1, \theta}}$.  Also, very often we will need to consider the following space of bounded linear operators  $\mathcal{L}(\mathbf{P}_{\mathbf{m}}^{\bm\varepsilon}X^\alpha_\varepsilon, \mathbf{Q}_{\mathbf{m}}^{\bm\varepsilon}X_\varepsilon^\alpha)$ and its norm will be abbreviated by $\|\cdot \|_{\mathcal{L}}$. 

\bigskip

%

%
%
%
%
%

\section{Smoothness of inertial manifolds}\label{smoothness-subsection}
\label{smoothness}
In this section we show the $C^{1,\theta}$ smoothness of the inertial manifolds $\Phi_\eps$ and $\Phi_0^\eps$ for a fixed value of the parameter $\eps$. Moreover,  we will obtain estimates of its $C^{1,\theta}$ norm which are independent of the parameter $\eps$.

\par\medskip Recall that the $C^1$ smoothness of the manifold is shown in \cite{Sell&You}, where they proved the following result:
 
\begin{teo}\label{smoothness}
Let hypotheses of Proposition \ref{existenciavariedadinercial} be satisfied. Assume that for each $\varepsilon>0$ the nonlinear functions $F_\varepsilon$, $F_0^\eps$ are Lipschitz $C^1$ functions from  $X_\varepsilon^\alpha$ to $X_\varepsilon$ and from $X_0^\alpha$ to $X_0$ respectively. Then, the inertial manifolds $\mathcal{M}_\varepsilon$, $\mathcal{M}_0^\varepsilon$ for $\varepsilon>0$, are $C^1$-manifolds and the functions  $\Psi_\varepsilon$, $\Psi_0^\varepsilon$ are Lipschitz $C^1$ functions from $\mathbf{P}_{\mathbf{m}}^{\bm\varepsilon}X_\varepsilon^\alpha$ to $\mathbf{Q}_{\mathbf{m}}^{\bm\varepsilon}X_\varepsilon^\alpha$ and from $\mathbf{P}_{\mathbf{m}}^{\bm 0}X_0^\alpha$ to $\mathbf{Q}_{\mathbf{m}}^{\bm 0}X_0^\alpha$.
\end{teo}

\begin{re}\label{relation-Phi-Psi}
i) Let us mention that the relation between the maps $\Psi_\varepsilon: \mathbf{P}_{\mathbf{m}}^{\bm\varepsilon}X_\varepsilon^\alpha\to \mathbf{Q}_{\mathbf{m}}^{\bm\varepsilon}X_\varepsilon^\alpha$ (resp.  $\Psi_0^\varepsilon: \mathbf{P}_{\mathbf{m}}^{\bm 0}X_0^\alpha\to \mathbf{Q}_{\mathbf{m}}^{\bm 0}X_0^\alpha$)  and $\Phi_\eps: \R^m\to  \mathbf{Q}_{\mathbf{m}}^{\bm\varepsilon}X_\varepsilon^\alpha$ (resp. $\Phi_0^\eps: \R^m\to  \mathbf{Q}_{\mathbf{m}}^{\bm 0}X_0^\alpha$) is $\Phi_\eps=\Psi_\eps\circ j_\eps^{-1}$ (resp. $\Phi_0^\eps=\Psi_0^\eps\circ j_0^{-1}$), where $j_\eps$ is defined by \eqref{definition-jeps}. 

\par\noindent ii) For the rest of the exposition, whenever we write  $\Psi_\eps$, $\Psi_0^\eps$, $\Phi_\eps$ and $\Phi_0^\eps$ we will refer to these maps that define the inertial manifolds. 

\end{re}

The proof of this theorem is based in 
the following extension of the Contraction Mapping Theorem, see \cite{ChowLuSell}.
\begin{lem}\label{contracion}
Let $X$ and $Y$ be complete metric spaces with metrics $d_x$ and $d_y$. Let $H: X\times Y\rightarrow X\times Y$ be a continuous function satisfying the following:
\begin{itemize}
\item[(1)] $H(x, y)=(F(x),G(x, y))$, $F$ does not depend on $y$.
\item[(2)] There is a constant $\theta$ with $0\leq\theta <1$ such that one has
$$d_x(F(x_1), F(x_2))\leq\theta d_x(x_1, x_2),\qquad x_1, x_2\in X,$$
$$d_y(G(x, y_1), G(x, y_2))\leq \theta d_y(y_1, y_2),\qquad x\in X, y_1, y_2\in Y.$$

\end{itemize}
Then there is a unique fixed point $(x^*, y^*)$ of $H$. Moreover, if $(x_n, y_n)$ is any sequence of iterations, 
$$(x_{n+1}, y_{n+1})=H(x_n, y_n)\qquad\textrm{for}\quad n\geq 1,$$
then
$$\lim_{n\rightarrow\infty}(x_n, y_n)=(x^*, y^*).$$

\end{lem}

In \cite{ChowLuSell} and \cite{Sell&You} the authors use this lemma to show the existence of an appropriate fixed point which will give the desired differentiability. In our case, we consider the maps 
$\mathbf{\Pi}_{\mathbf{0}}^\eps:\tilde{\mathcal{F}}_0(L,R)\times \mathcal{E}_0\rightarrow \mathcal{\tilde F}_0(L,R)\times \mathcal{E}_0$ and 
$\mathbf{\Pi}_{\bm\eps}:\mathcal{\tilde F}_\eps(L,R)\times \mathcal{E}_\eps\rightarrow \mathcal{\tilde F}_\eps(L,R)\times \mathcal{E}_\eps$ 
given by
 $$\mathbf{\Pi}_{\mathbf{0}}^\eps: (\upchi^\eps_0, \Upsilon^\eps_0)\rightarrow (\mathbf{T}^\eps_{\mathbf{0}} \upchi^\eps_0,\mathbf{D}_0^\eps(\upchi^\eps_0, \Upsilon^\eps_0)),$$
 and
 $$\mathbf{\Pi}_{\bm\eps}: (\upchi_\varepsilon, \Upsilon_\varepsilon)\rightarrow (\mathbf{T}_{\bm\varepsilon} \upchi_\varepsilon, \mathbf{D}_\eps(\upchi_\varepsilon, \Upsilon_\varepsilon)),$$
 where 
  $$\mathcal{\tilde F}_\eps(L,R)\mathord=\Big\{\upchi_\varepsilon:\mathbf{P}_{\mathbf{m}}^{\bm\varepsilon}X^\alpha_\varepsilon\rightarrow\mathbf{Q}_{\mathbf{m}}^{\bm\varepsilon}X^\alpha_\varepsilon \,\,/\,\, \|\upchi_\varepsilon(p)-\upchi_\varepsilon(p')\|_{X^\alpha_\varepsilon}\leq L\|p-p'\|_{X_\eps^\alpha}, \,\,\, p, p'\in \mathbf{P}_{\mathbf{m}}^{\bm\varepsilon}X_\eps^\alpha,$$
 $$\hbox{supp}(\upchi_\eps)\subset \{ \phi\in \mathbf{P}_{\mathbf{m}}^{\bm\eps} X_\eps^\alpha, \|\phi\|_{X_\eps^\alpha}\leq R\}\Big\}, \qquad 0\leq \eps\leq\eps_0$$
 and 
 $$\mathcal{E}_\eps=\{\Upsilon_\eps:\mathbf{P}_{\mathbf{m}}^{\bm\varepsilon}X^\alpha_\varepsilon\rightarrow\mathcal{L}(\mathbf{P}_{\mathbf{m}}^{\bm\varepsilon}X^\alpha_\varepsilon, \mathbf{Q}_{\mathbf{m}}^{\bm\varepsilon}X_\varepsilon^\alpha)\hbox{ continuous}:\,\,\,\, \qquad\qquad$$
 $$\qquad\qquad\qquad\|\Upsilon_\eps(p)p'\|_{X_\varepsilon^\alpha}\leq  \|p'\|_{X_\eps^\alpha},\quad p, p'\in\mathbf{P}_{\mathbf{m}}^{\bm\varepsilon}X^\alpha_\varepsilon \}\qquad 0\leq \eps\leq\eps_0.$$
Notice that the last contiditon in the definition of $\mathcal{E}_\eps$ could be written equivalently as $\|\Upsilon_\eps(p)\|_{\mathcal{L}}\leq 1$
for all  $p\in \mathbf{P}_{\mathbf{m}}^{\bm\varepsilon}X^\alpha_\varepsilon$. 
\bigskip

The functionals $\mathbf{T}^{\bm\eps}_{\mathbf{0}}$, $\mathbf{T}_{\bm\varepsilon}$ are the ones used in the Lyapunov-Perron method to prove the existence of the inertial manifolds, see \cite{Sell&You}, which are defined as

\begin{equation}\label{definition-T0Psi2}
 (\mathbf{T}^{\bm\eps}_{\mathbf{0}}\upchi^\eps_0)(\xi)=\int_{-\infty}^0e^{A_\varepsilon\mathbf{Q}^{\mathbf{0}}_{\mathbf{m}} s}\mathbf{Q}^{\mathbf{0}}_{\mathbf{m}} F^\eps_0(u^\eps_0(s))ds,
 \end{equation}
\begin{equation}\label{definition-TepsPsi2}
 (\mathbf{T}_{\bm \varepsilon}\upchi_\varepsilon)(\eta)=\int_{-\infty}^0e^{A_\varepsilon\mathbf{Q}^{\bm\varepsilon}_{\mathbf{m}} s}\mathbf{Q}^{\bm\eps}_{\mathbf{m}} F_\eps(u_\eps(s))ds,
 \end{equation}
 with $u^\eps_0(t)=p^\eps_0(t)+\upchi^\eps_0(p^\eps_0(t))$, $u_\varepsilon(t)=p_\varepsilon(t)+\upchi_\varepsilon(p_\varepsilon(t))$, where $p_0^\eps(\cdot)\in [\varphi_1^0,\ldots,\varphi_m^0]$  is the globally defined solution of 
\begin{equation}\label{equationp*}
\left\{
\begin{array}{l}
p_t=-A_0 p+\mathbf{P}_{\mathbf{m}}^{\mathbf{0}} F^\eps_0(p+\upchi^\eps_0(p(t)))\\
p(0)=\xi\in [\varphi_1^0,\ldots,\varphi_m^0]
\end{array}
\right.
\end{equation}
and  $p_\eps(\cdot)\in [\varphi_1^\eps,\ldots,\varphi_m^\eps]$ is the globally defined solution of 
\begin{equation}\label{equationp}
\left\{
\begin{array}{l}
p_t=-A_\varepsilon p+\mathbf{P}_{\mathbf{m}}^{\bm\eps} F_\eps(p+\upchi_\eps(p(t)))\\
p(0)=\eta\in [\varphi_1^\eps,\ldots,\varphi_m^\eps]. 
\end{array}
\right.
\end{equation}


The functionals, ${\mathbf{D}^{\bm\eps}_{\mathbf{0}}}(\upchi^\eps_0, \Upsilon^\eps_0)$, $\mathbf{D}_{\bm\varepsilon}(\upchi_\varepsilon, \Upsilon_\varepsilon)$ are given as follows: for any $\xi\in \mathbf{P}_{\mathbf{m}}^{\mathbf{0}}X^\alpha_0$, $\eta\in \mathbf{P}_{\mathbf{m}}^{\bm\varepsilon}X^\alpha_\varepsilon$,

\begin{equation}\label{differential-Psi*} 
{\mathbf{D}^{\bm\eps}_{\mathbf{0}}}(\upchi^\eps_0, \Upsilon^\eps_0)(\xi)=\int_{-\infty}^0 e^{A_0 \mathbf{Q}_{\mathbf{m}}^{\mathbf{0}} s} \mathbf{Q}_{\mathbf{m}}^{\mathbf{0}} DF^\eps_0(u^\eps_0(s))(I+\Upsilon^\eps_0(p^\eps_0(s)))\Theta^\eps_0(\xi,s)ds,
\end{equation}

and
\begin{equation}\label{differential-Psi} 
\mathbf{D}_{\bm\varepsilon}(\upchi_\varepsilon, \Upsilon_\varepsilon)(\eta)=\int_{-\infty}^0 e^{A_\varepsilon \mathbf{Q}_{\mathbf{m}}^{\bm\varepsilon} s} \mathbf{Q}_{\mathbf{m}}^{\bm\varepsilon} DF_\varepsilon(u_\varepsilon(s))(I+\Upsilon_\varepsilon(p_\varepsilon(s)))\Theta_\varepsilon(\eta,s)ds,
\end{equation}
with $u^\eps_0$, $p^\eps_0$, $u_\eps$, $p_\eps$ as above and moreover,  $\Theta^\eps_0(\xi, t)=\Theta^\eps_0(\upchi^\eps_0, \Upsilon^\eps_0, \xi, t)$, $\Theta_\varepsilon(\eta, t)=\Theta_\varepsilon(\upchi_\eps, \Upsilon_\eps, \eta, t)$ are the linear maps from $\mathbf{P}_{\mathbf{m}}^{\mathbf{0}}X_0^\alpha$ to $\mathbf{P}_{\mathbf{m}}^{\mathbf{0}}X_0^\alpha$ and from $\mathbf{P}_{\mathbf{m}}^{\bm\varepsilon}X_\varepsilon^\alpha$ to $\mathbf{P}_{\mathbf{m}}^{\bm\varepsilon}X_\varepsilon^\alpha$ satisfying 
\begin{equation}\label{equationTheta*-section6}
\left\{
\begin{array}{l}
{\Theta}_t= -A_0 \Theta+\mathbf{P}_{\mathbf{m}}^{\mathbf{0}}  DF^\eps_0(u^\eps_0(t))(I+\Upsilon^\eps_0(p^\eps_0(t)))\Theta\\
\Theta(\xi,0)=I,
\end{array}
\right.
\end{equation}
and
\begin{equation}\label{equationTheta-section6}
\left\{
\begin{array}{l}
\Theta_t= -A_\varepsilon \Theta+\mathbf{P}_{\mathbf{m}}^{\bm\varepsilon}  DF_\varepsilon(u_\varepsilon(t))(I+\Upsilon_\varepsilon(p_\varepsilon(t)))\Theta\\
\Theta(\eta,0)=I,
\end{array}
\right.
\end{equation}
respectively. 

In fact, in these works it is obtained that the fixed point of the maps $\mathbf{\Pi}_{\mathbf{0}}^\eps$ and $\mathbf{\Pi}_\eps$ are given by  $({\upchi^\eps_0}^*, {\Upsilon^\eps_0}^*)=(\Psi^{\eps}_0, D\Psi^{\eps}_0)$, $(\upchi_\varepsilon^*, \Upsilon_\varepsilon^*)=(\Psi_\varepsilon, D\Psi_\varepsilon)$ with $\Psi^\eps_0$ and $\Psi_\varepsilon$ are the maps whose graphs gives us the inertial manifolds (see Remark \ref{relation-Phi-Psi} ii)), which are given by the fixed points of the functionals $\mathbf{T}^{\bm\eps}_{\mathbf{0}}$ and $\mathbf{T}_{\bm\varepsilon}$ and $D\Psi_0^\varepsilon$, $D\Psi_\varepsilon$ are the Frechet derivatives of the inertial manifolds.

\bigskip 


\par\medskip In order to prove the $C^{1,\theta}$ smoothness of the inertial manifolds $\Phi_0^\eps$, $\Phi_\eps$, we will show that if we denote the set
$$\mathcal{E}_{\eps}^{ \theta,M}=\{\Upsilon_\eps\in \mathcal{E}_{\eps} : \|\Upsilon_\eps(p)-\Upsilon_\eps(p')\|_{\mathcal{L}}
\leq M\|p-p'\|_{X_\eps^\alpha}^\theta,\quad \forall p, p'\in\mathbf{P}_{\mathbf{m}}^{\bm\eps}X_\eps^\alpha\}$$
which is a closed set in $\mathcal{E}_{\eps}$, then there exist appropriate $\theta$ and $M$ such that the maps ${\mathbf{D}^{\bm\eps}_{\mathbf{0}}}(\Psi^\eps_0, \cdot)$ and $\mathbf{D}_{\bm\varepsilon}(\Psi_\varepsilon, \cdot)$ from \eqref{differential-Psi*} and \eqref{differential-Psi} with $\Psi^\eps_0$, $\Psi_\varepsilon$ the obtained inertial manifolds, transform $\mathcal{E}_{\eps}^{ \theta,M}$ into itself, see Lemma \ref{PsiUniform} below, which will imply that the fixed point of the maps $\mathbf{\Pi}_{\mathbf{0}}^\eps$ and $\mathbf{\Pi}_{\bm\eps}$ lie in $\mathcal{\tilde F}_0(L,R)\times \mathcal{E}_0^{ \theta,M}$ and $\mathcal{\tilde F}_\eps(L,R)\times \mathcal{E}_\eps^{ \theta,M}$, respectively,   obtaining the desired regularity.

%
%

\bigskip

Throughout this subsection, we provide a proof of Proposition \ref{FixedPoint-E^1Theta} for the inertial manifold $\Phi_\eps$ for each $\eps\geq 0$. 
Note that the proof of this result for the inertial manifold $\Phi^\eps_0$, consists in following, step by step, the same proof. Then, we focus now in the inertial manifold $\Phi_\eps$ with $\eps>0$ fixed.  
\bigskip

We start with some estimates. 
\begin{lem}\label{distp-section6}
Let $p^1_\varepsilon(t)$ and $p^2_\eps(t)$ be solutions of (\ref{equationp}) with $p^1_\eps(0)$ and $p^2_\eps(0)$ its initial data, respectively. Then,  for $t\leq 0$,
$$\|p^1_\eps(t)-p^2_\eps(t)\|_{X_\eps^\alpha}\leq \|p_\eps^1(0)-p_\eps^2(0)\|_{X_\eps^\alpha} e^{-[2L_F(\lambda_m^\eps)^\alpha+\lambda_m^\eps]t}$$
\end{lem}
\begin{proof}
By the variation of constants formula,
$$p_\eps^1(t)-p_\eps^2(t)=e^{-A_\eps t}[p_\eps^1(0)-p_\eps^2(0)]+$$
$$+\int_0^te^{-A_\eps(t-s)}\mathbf{P}_{\mathbf{m}}^{\bm\eps}[F_\eps(p_\eps^1(s)+\Psi_\eps(p_\eps^1(s)))-F_\eps(p_\eps^2(s)+\Psi_\eps(p_\eps^2(s)))]ds.$$
Hence, applying \eqref{semigrupoproyectadoP} and \eqref{semigrupoproyectadoP-alpha} and taking into account that  $\Psi_\varepsilon, F_\varepsilon$ are uniformly Lipschitz with Lipschitz constants $L<1$ and $L_F$, respectively, we get

$$\|p_\eps^1(t)-p_\eps^2(t)\|_{X_\eps^\alpha}\leq e^{-\lambda_m^\eps t}\|p_\eps^1(0)-p_\eps^2(0)\|_{X_\eps^\alpha}+2L_F(\lambda_m^\eps)^\alpha\int_t^0e^{-\lambda_m^\eps(t-s)}\|p_\eps^1(s)-p_\eps^2(s)\|_{X_\eps^\alpha}ds. $$
By Gronwall inequality,
$$\|p_\eps^1(t)-p_\eps^2(t)\|_{X_\eps^\alpha}\leq \|p_\eps^1(0)-p_\eps^2(0)\|_{X_\eps^\alpha}e^{-[2L_F(\lambda_m^\eps)^\alpha+\lambda_m^\eps]t},$$
as we wanted to prove.
\end{proof}

\begin{lem}\label{Jnorm}
Let $\Psi_\varepsilon\in\mathcal{\tilde F}_\eps(L, R)$ with $L<1$ and $\Upsilon_\eps\in \mathcal{E}_\eps$, $0<\varepsilon\leq\eps_0$. Then, for $t\leq 0$,
$$\|\Theta_\varepsilon(p_\varepsilon^0, t)\|_{\mathcal{L}}
\leq  e^{-[2L_F(\lambda_m^\varepsilon)^\alpha+\lambda_m^\varepsilon]t}.$$
\end{lem}

\begin{proof}
If $z_\eps\in\mathbf{P}_{\mathbf{m}}^{\bm\varepsilon}X_\varepsilon^\alpha$, with the aid of the variation of constants formula applied to (\ref{equationTheta-section6}), we have for $t\leq 0$, 
$$\|\Theta_\varepsilon(p_\varepsilon^0, t)z_\varepsilon\|_{X_\varepsilon^\alpha}\leq \|e^{-A_\varepsilon\mathbf{P}_{\mathbf{m}}^{\bm\varepsilon}t} z_\varepsilon\|_{X_\varepsilon^\alpha}+$$
$$+\int_t^0\left\|e^{-A_\varepsilon\mathbf{P}_{\mathbf{m}}^{\bm\varepsilon}(t-s)}\mathbf{P}_{\mathbf{m}}^{\bm\varepsilon}DF_\varepsilon(u_\varepsilon(s))(I+\Upsilon_\varepsilon(p_\varepsilon(s)))\Theta_\varepsilon(p_\varepsilon^0, s)z_\varepsilon\right\|_{X_\varepsilon^\alpha}ds.$$
Hence as before, 
$$\|\Theta_\varepsilon(p_\varepsilon^0, t)z_\varepsilon\|_{X_\varepsilon^\alpha}\leq e^{-\lambda_m^\varepsilon t}\|z_\eps\|_{X_\eps^\alpha}+2L_F (\lambda_m^\varepsilon)^\alpha\int_t^0e^{-\lambda_m^\varepsilon (t-s)}\|\Theta_\varepsilon(p_\varepsilon^0, s)z_\varepsilon\|_{X_\varepsilon^\alpha}.$$
Using Gronwall inequality, we get 
$$\|\Theta_\varepsilon(p_\varepsilon^0, t)z_\eps\|_{X_\varepsilon^\alpha}\leq e^{-[2L_F(\lambda_m^\varepsilon)^\alpha+\lambda_m^\varepsilon]  t}\|z_\eps\|_{X_\varepsilon^\alpha}$$
from where we get the result. 
\end{proof}

\begin{lem}\label{distThetaEpsilon}
Let  $0<\theta\leq \theta_F$  and $M>0$ fixed. Let $p_\eps^1, p_\eps^2\in\mathbf{P}_{\mathbf{m}}^{\bm\eps} X_\eps^\alpha$ and consider  $\Theta_\eps^1(t)=\Theta_\eps(p_\eps^1, t)$,  $\Theta_\eps^2(t)=\Theta_\eps(p_\eps^2, t)$ the  solutions of (\ref{equationTheta-section6}) for some  $\Upsilon_\eps\in \mathcal{E}_\eps^{\theta,M}$. Then, for $t\leq 0$,
$$\|\Theta_\eps^1(t)-\Theta_\eps^2(t)\|_{\mathcal{L}}
\leq \left(\frac{2L}{(\theta+1)L_F}+\frac{M}{2(\theta+1)}\right)\|p_\eps^1-p_\eps^2\|_{X_\eps^\alpha}^\theta\,\, e^{-(2(\theta+2)L_F(\lambda_m^\eps)^\alpha +(\theta+1)\lambda_m^\eps) t}.$$  
\end{lem}
\begin{proof}
Applying the variation of constants formula to (\ref{equationTheta-section6}), for $t\leq 0$,
$$\|\Theta_\eps^1(t)-\Theta_\eps^2(t)\|_{\mathcal{L}}\leq 
\int_t^0\Big\|e^{-A_\eps\mathbf{P}_{\mathbf{m}}^{\bm\eps}(t-s)}\mathbf{P}_{\mathbf{m}}^{\bm\eps}[DF_\eps(u_\eps^1(s))(I+\Upsilon_\eps(p_\eps^1(s)))\Theta^1_\eps(s)$$
$$-DF_\eps(u_\eps^2(s))(I+\Upsilon_\eps(p_\eps^2(s)))\Theta^2_\eps(s)]\Big\|_{\mathcal{L}}ds$$
with $u_\eps^i(s)=p_\eps^i(s)+\Psi_\eps(p_\eps^i(s))$, $i= 1, 2$.

We can decompose the above integral in the following way,
$$\|\Theta_\eps^1(t)-\Theta_\eps^2(t)\|_{\mathcal{L}}\leq$$
$$\leq\int_t^0\Big\|e^{-A_\eps\mathbf{P}_{\mathbf{m}}^{\bm\eps}(t-s)}\mathbf{P}_{\mathbf{m}}^{\bm\eps} [DF_\eps(u_\eps^1(s))-DF_\eps(u_\eps^2(s))](I+\Upsilon_\eps(p_\eps^1(s)))\Theta_\eps^1(s)\Big\|_{\mathcal{L}}ds+$$
$$+\int_t^0\Big\|e^{-A_\eps\mathbf{P}_{\mathbf{m}}^{\bm\eps}(t-s)}\mathbf{P}_{\mathbf{m}}^{\bm\eps} DF_\eps(u_\eps^2(s))(\Upsilon_\eps(p_\eps^1(s))-\Upsilon_\eps(p_\eps^2(s)))\Theta_\eps^1(s)\Big\|_{\mathcal{L}}ds+$$
$$+\int_t^0\Big\|e^{-A_\eps\mathbf{P}_{\mathbf{m}}^{\bm\eps}(t-s)}\mathbf{P}_{\mathbf{m}}^{\bm\eps} DF_\eps(u_\eps^2(s))[(I+\Upsilon_\eps(p_\eps^2(s)))(\Theta_\eps^1(s)-\Theta_\eps^2(s))\Big\|_{\mathcal{L}}ds=$$
$$=I_1+I_2+I_3.$$
We analyze each term separately. 

By hipothesis {\bf(H2')}, \eqref{semigrupoproyectadoP-alpha} and Lemma \ref{Jnorm}, 
$$I_1\leq 2 L(\lambda_m^\eps)^\alpha e^{-\lambda_m^\eps t}\int_t^0 \|u^1_\eps(s)-u^2_\eps(s)\|_{X_\eps^\alpha}^\theta e^{-2L_F(\lambda_m^\eps)^\alpha s}ds\leq $$
$$\leq 4L(\lambda_m^\eps)^\alpha e^{-\lambda_m^\eps t}\int_t^0 \|p_\eps^1(s)-p_\eps^2(s)\|_{X_\eps^\alpha}^\theta e^{-2L_F(\lambda_m^\eps)^\alpha s}ds.$$
Applying Lemma \ref{distp-section6},
$$I_1\leq \frac{2L}{(\theta+1)L_F}\|p_\eps^1-p_\eps^2\|_{X_\eps^\alpha}^\theta e^{-[2(\theta+1)L_F(\lambda_m^\eps)^\alpha+(\theta+1)\lambda_m^\eps]t}.$$
Since $\Upsilon_\eps\in \mathcal{E}_\eps^{\theta,M}$, $0<\theta\leq\theta_F$, and by Lemma \ref{Jnorm}, we have
$$I_2\leq  L_F(\lambda_m^\eps)^\alpha Me^{-\lambda_m^\eps t}\int_t^0 \|p^1_\eps(s)-p^2_\eps(s)\|_{X_\eps^\alpha}^\theta e^{-2L_F(\lambda_m^\eps)^\alpha s} ds.$$
Applying Lemma \ref{distp-section6},
$$I_2\leq\frac{M}{2(\theta+1)}\|p_\eps^1-p_\eps^2\|_{X_\eps^\alpha}^\theta e^{-[2(\theta+1)L_F(\lambda_m^\eps)^\alpha + (\theta+1)\lambda_m^\eps]t}.$$
This last term is estimated as follows, 
$$I_3\leq 2L_F(\lambda_m^\eps)^\alpha\int_t^0e^{-\lambda_m^\eps(t- s)}\|\Theta_\eps^1(s)-\Theta_\eps^2(s)\|_{\mathcal{L}}ds.$$

So,
$$\|\Theta_\eps^1(t)-\Theta_\eps^2(t)\|_{\mathcal{L}}\leq $$
$$ \left(\frac{2L}{(\theta+1)L_F}+\frac{M}{2(\theta+1)}\right)\|p_\eps^1-p_\eps^2\|_{X_\eps^\alpha}^\theta e^{-[2(\theta+1)L_F(\lambda_m^\eps)^\alpha+(\theta+1)\lambda_m^\eps]t}$$
$$+2L_F(\lambda_m^\eps)^\alpha\int_t^0e^{-\lambda_m^\eps (t-s)}\|\Theta_\eps^1(s)-\Theta_\eps^2(s)\|_{\mathcal{L}}ds.$$
Applying Gronwall inequality,
$$\|\Theta_\eps^1(t)-\Theta_\eps^2(t)\|_{\mathcal{L}}\leq $$
$$\leq \left(\frac{2L}{(\theta+1)L_F}+\frac{M}{2(\theta+1)}\right)\|p_\eps^1-p_\eps^2\|_{X_\eps^\alpha}^\theta e^{-[2(\theta+2)L_F(\lambda_m^\eps)^\alpha+(\theta+1)\lambda_m^\eps]t},$$
which shows the result.
\end{proof}

\bigskip

For the sake of notation, there are several exponents that repeat themselves very often and they are kind of long. We will abbreviate the exponents as follows:

\begin{equation}\label{def-exponents}
\begin{array}{l}
\Lambda_0=2L_F(\lambda_m^\eps)^\alpha +\lambda_m^\eps \\
\Lambda_1=\lambda_{m+1}^\eps -(\theta+1)\lambda_m^\eps-2(\theta+1) L_F(\lambda_m^\eps)^\alpha \\
\Lambda_2=\lambda_{m+1}^\eps -(\theta+1)\lambda_m^\eps-2(\theta+2) L_F(\lambda_m^\eps)^\alpha
\end{array}
\end{equation}

We can prove now the following Lemma.
\begin{lem}\label{PsiUniform}
If we choose $\theta$ such that $0<\theta\leq\theta_F$ and $\theta< \theta_0$ with $\theta_0$ given by \eqref{theta}, then there exist $M_0=M_0(\theta)>0$ such that for each $M\geq M_0$ and for $\eps$ small enough, we have 
$\mathbf{D}_{\bm\eps}(\Psi_\eps, \cdot)$ maps $\mathcal{E}_\eps^{\theta,M}$ into $\mathcal{E}_\eps^{\theta,M}$. 
\end{lem}


\bigskip

\begin{proof}
Let $\Upsilon_\eps\in \mathcal{E}_\eps^{\theta,M}$ and $p_\eps^1, p_\eps^2\in\mathbf{P}_{\mathbf{m}}^{\bm\eps}X_\eps^\alpha$. In \cite{Sell&You} the authors prove $\mathbf{D}_{\bm\eps}(\Psi_\eps, \cdot)$ maps $\mathcal{E}_\eps$ into $\mathcal{E}_\eps$. So, it remains to prove that,
$$\|\mathbf{D}_{\bm\eps}(\Psi_\eps, \Upsilon_\eps)(p_\eps^1)-\mathbf{D}_{\bm\eps}(\Psi_\eps, \Upsilon_\eps)(p_\eps^2)\|_{\mathcal{L}}\leq M\|p_\eps^1-p_\eps^2\|_{X_\eps^\alpha}^\theta,$$
with $M$ and $\theta$ as in the statement. 

From expression (\ref{differential-Psi}), we have,
$$\|\mathbf{D}_{\bm\eps}(\Psi_\eps, \Upsilon_\eps)(p_\eps^1)-\mathbf{D}_{\bm\eps}(\Psi_\eps, \Upsilon_\eps)(p_\eps^2)\|_{\mathcal{L}}\leq$$
$$\int_{-\infty}^0\Big\|e^{A_\eps\mathbf{Q}_{\mathbf{m}}^{\bm\eps}s}\mathbf{Q}_{\mathbf{m}}^{\bm\eps}[DF_\eps(u_\eps^1(s))(I+\Upsilon_\eps(p_\eps^1(s)))\Theta^1_\eps(s)-DF_\eps(u_\eps^2(s))(I+\Upsilon_\eps(p_\eps^2(s)))\Theta^2_\eps(s)]\Big\|_{\mathcal{L}}ds,$$
with $p_\eps^i(s)$ the solution of (\ref{equationp}) with $p_\eps^i(0)=p_\eps^i$ and $u_\eps^i(s)=p_\eps^i(s)+\Psi_\eps(p_\eps^i(s))$, for $i=1, 2$.

In a similar way as in proof of Lemma \ref{distThetaEpsilon}, we decompose it as follows,
$$\|\mathbf{D}_{\bm\eps}(\Psi_\eps, \Upsilon_\eps)(p_\eps^1)-\mathbf{D}_{\bm\eps}(\Psi_\eps, \Upsilon_\eps)(p_\eps^2)\|_{\mathcal{L}}\leq$$
$$\leq\int_{-\infty}^0\Big\|e^{A_\eps\mathbf{Q}_{\mathbf{m}}^{\bm\eps}s}\mathbf{Q}_{\mathbf{m}}^{\bm\eps} [DF_\eps(u_\eps^1(s))-DF_\eps(u_\eps^2(s))](I+\Upsilon_\eps(p_\eps^1(s)))\Theta_\eps^1(s)\Big\|_{\mathcal{L}}ds+$$
$$+\int_{-\infty}^0\Big\|e^{A_\eps\mathbf{Q}_{\mathbf{m}}^{\bm\eps}s}\mathbf{Q}_{\mathbf{m}}^{\bm\eps} DF_\eps(u_\eps^2(s))[\Upsilon_\eps(p_\eps^1(s))-\Upsilon_\eps(p_\eps^2(s))]\Theta_\eps^1(s)\Big\|_{\mathcal{L}}ds+$$
$$+\int_{-\infty}^0\Big\|e^{A_\eps\mathbf{Q}_{\mathbf{m}}^{\bm\eps}s}\mathbf{Q}_{\mathbf{m}}^{\bm\eps} DF_\eps(u_\eps^2(s))(I+\Upsilon_\eps(p_\eps^2(s)))[\Theta_\eps^1(s)-\Theta_\eps^2(s)]\Big\|_{\mathcal{L}}ds=$$
$$=I_1+I_2+I_3.$$
Following the same arguments used in that proof and since $\Upsilon_\eps\in \mathcal{E}_\eps^{\theta,M}$ we get
$$I_1\leq 4L(\lambda_{m+1}^\eps)^\alpha\|p_\eps^1-p_\eps^2\|_{X_\eps^\alpha}^\theta\int_{-\infty}^0 e^{\Lambda_1 s}ds\leq\frac{4L(\lambda_{m+1}^\eps)^\alpha}{\Lambda_1}\|p_\eps^1-p_\eps^2\|_{X_\eps^\alpha}^\theta$$
Similarly, for $I_2$,
$$I_2\leq L_F(\lambda_{m+1}^\eps)^\alpha M\|p_\eps^1-p_\eps^2\|_{X_\eps^\alpha}^\theta\int_{-\infty}^0 e^{\Lambda_1 s}ds\leq 
\frac{L_F(\lambda_{m+1}^\eps)^\alpha M}{\Lambda_1}\|p_\eps^1-p_\eps^2\|_{X_\eps^\alpha}^\theta.$$
And finally, applying Lemma \ref{distThetaEpsilon},
$$I_3\leq 2L_F(\lambda_{m+1}^\eps)^\alpha \left(\frac{2L}{(\theta+1)L_F}+\frac{M}{2(\theta+1)}\right)\|p_\eps^1-p_\eps^2\|_{X_\eps^\alpha}^\theta\int_{-\infty}^0 e^{-\Lambda_2 s}ds$$
which implies,
$$I_3\leq\frac{2L_F(\lambda_{m+1}^\eps)^\alpha}{\Lambda_2} \left(\frac{2L}{(\theta+1)L_F}+\frac{M}{2(\theta+1)}\right)\|p_\eps^1-p_\eps^2\|_{X_\eps^\alpha}^\theta.$$
\bigskip

Putting everything together we obtain

$$\|\mathbf{D}_{\bm\eps}(\Psi_\eps, \Upsilon_\eps)(p_\eps^1)-\mathbf{D}_{\bm\eps}(\Psi_\eps, \Upsilon_\eps)(p_\eps^2)\|_{\mathcal{L}}\leq$$
$$(4L+ML_F)(\lambda_{m+1}^\eps)^\alpha \Big(\frac{1}{\Lambda_1}+\frac{1}{(\theta+1)\Lambda_2}\Big)\|p_\eps^1-p_\eps^2\|_{X_\eps^\alpha}^\theta$$

%
%
%
But since $\Lambda_2\leq \Lambda_1$, see \eqref{def-exponents}, and $\theta>0$, we have
$$\|\mathbf{D}_{\bm\eps}(\Psi_\eps, \Upsilon_\eps)(p_\eps^1)-\mathbf{D}_{\bm\eps}(\Psi_\eps, \Upsilon_\eps)(p_\eps^2)\|_{\mathcal{L}}\leq
(4L+ML_F)(\lambda_{m+1}^\eps)^\alpha\frac{2}{\Lambda_2}\|p_\eps^1-p_\eps^2\|_{X_\eps^\alpha}^\theta$$
$$=\Big( \frac{8L (\lambda_{m+1}^\eps)^{\alpha}}{\Lambda_2}+M\frac{2L_F(\lambda_{m+1}^\eps)^\alpha}{\Lambda_2}\Big)\|p_\eps^1-p_\eps^2\|_{X_\eps^\alpha}^\theta$$

But if we consider
$$\theta^0=\frac{\lambda_{m+1}^0-\lambda_m^0-4L_F(\lambda_m^0)^\alpha-2L_F(\lambda_{m+1}^0)^\alpha}{2L_F(\lambda_m^0)^\alpha+\lambda_m^0},$$
then, direct computations show that if $\theta<\theta_0$ and $\eps$ is small, then  $\frac{2L_F(\lambda_{m+1}^\eps)^\alpha}{\Lambda_2}\leq \eta$ for some $\eta<1$. This implies that if we choose $M$ large enough then 
$$\Big( \frac{8L (\lambda_{m+1}^\eps)^{\alpha}}{\Lambda_2}+M\frac{2L_F(\lambda_{m+1}^\eps)^\alpha}{\Lambda_2}\Big)\leq M$$
which shows the result. 
%
%
\end{proof}

\bigskip

We can prove now the main result of this subsection.

\begin{proof} {\sl (of Proposition \ref{FixedPoint-E^1Theta})} 
Again, we do only the proof for $\Phi_\eps$ being the proof for $\Phi_0^\eps$ completely similar.  

Since $\Phi_\eps=\Psi_\eps\circ j_\eps^{-1}$ and $j_\eps$ is an isomorphism, see Remark \ref{relation-Phi-Psi} and (\ref{definition-jeps}), to prove $\Phi_\eps\in C^{1,\theta}(\mathbb{R}^m, X_\eps^\alpha)$ for some $\theta$, is equivalent to prove $\Psi_\eps\in C^{1,\theta}(\mathbf{P}_{\mathbf{m}}^{\bm\eps} X_\eps^\alpha, X_\eps^\alpha)$.

In \cite{Sell&You}, the authors prove the existence of the unique fixed point $(\Psi_\eps^*, \Upsilon_\eps^*)=(\Psi_\eps, D\Psi_\eps)\in\mathcal{\tilde F}_\eps(L, R)\times \mathcal{E}_\eps$ of the map
$$\mathbf{\Pi}_{\bm\eps}: (\Psi_\eps, \Upsilon_\eps)\rightarrow (\mathbf{T}_{\bm\eps}\Psi_\eps, \mathbf{D}_{\bm\eps}(\Psi_\eps, \Upsilon_\eps)).$$  We want to prove that, in fact, this fixed point belongs to $\mathcal{\tilde F}_\eps(L, R)\times \mathcal{E}_\eps^{\theta,M}$. We proceed as follows. Let $\{z_n\}_{n\geq 0}\in  \mathcal{\tilde F}_\eps(L, R)\times \mathcal{E}_\eps^{\theta,M}$ be a sequence given by
$$z_0= (\Psi_\eps, 0),\qquad z_1=\mathbf{\Pi}_{\bm\eps} z_1=(\mathbf{T}_{\bm\eps}\Psi_\eps, \mathbf{D}_{\bm\eps}(\Psi_\eps, 0)),\quad ... \quad z_n=\mathbf{\Pi}_{\bm\eps}^nz_0.$$
Note that the first coordinate of $z_n$ is $\mathbf{T}_{\bm\eps}^n\Psi_\eps$ which coincides with $\Psi_\eps$ for all $n=1,2,\ldots$ since $\Psi_\eps$ is fixed point of $\mathbf{T}_{\bm\eps}$.   Hence, by Lemma \ref{PsiUniform}, $\{z_n\}_{n\geq 0}\in  \mathcal{\tilde F}_\eps(L, R)\times \mathcal{E}_\eps^{\theta,M}$ with $\theta$ and $M$ described in this lemma.

By Lemma \ref{contracion},
$$\lim_{n\rightarrow\infty} z_n=  (\Psi_\eps, D\Psi_\eps).$$
Hence, since $\mathcal{E}_\eps^{\theta,M}$ is a closed subspace of $\mathcal{E}_\eps$ and $z_n\in \mathcal{E}_\eps^{\theta,M}$ for all $n=1,2,\ldots$, then 
$$ (\Psi_\eps, D\Psi_\eps)\in  \mathcal{\tilde F}_\eps(L, R)\times \mathcal{E}_\eps^{\theta,M}.$$
That is, $\Psi_\eps\in C^{1,\theta}(\mathbf{P}_{\mathbf{m}}^{\bm\eps} X_\eps^\alpha, X_\eps^\alpha)$, for $0<\eps\leq\eps_0$, with $0<\theta\leq\theta_F$ and $\theta<\theta^0$, see (\ref{theta}). Then, $\Phi_\eps\in C^{1,\theta}(\mathbb{R}^m, X_\eps^\alpha)$ as we wanted to prove.
\end{proof}
\par\bigskip\bigskip

%
%

\bigskip

\section{$C^{1, \theta}$-estimates on the inertial manifolds}
\label{convergence}

In this section we study the $C^{1, \theta}$-convergence, with $0<\theta\leq 1$ small enough, of the inertial manifolds $\Phi^\eps_0$, $\Phi_\eps$, $0<\eps\leq\eps_0$. For that we will obtain first the $C^1$-convergence of these manifolds, and, with an interpolation argument and applying the results obtained in the previous subsection, we get the $C^{1, \theta}$-convergence and a rate of this convergence.

\bigskip

Before proving the main result of this subsection, Theorem \ref{convergence-C^1-theo}, we need the following estimate.
\begin{lem}\label{Jdistance}
Let $\Theta^\eps_0(j_0^{-1}(z),t)=\Theta^\eps_0(\Psi_0^\eps, D\Psi_0^\eps, j_0^{-1}(z),t)$ and $\Theta_\varepsilon(j_\varepsilon^{-1}(z),t)=\Theta_\varepsilon(\Psi_\eps, D\Psi_\eps, j_\varepsilon^{-1}(z),t)$ be solutions of (\ref{equationTheta*-section6}) and (\ref{equationTheta-section6}), for $z\in\mathbb{R}^m$ and $t\leq 0$. Then, we have,
$$\|\mathbf{P}_{\mathbf{m}}^{\bm\varepsilon}E\Theta^\eps_0(j_0^{-1}(z), t)-\Theta_\varepsilon(j_\varepsilon^{-1}(z),t)\mathbf{P}_{\mathbf{m}}^{\bm\varepsilon}E\|_{\mathcal{L}(\mathbf{P}_{\mathbf{m}}^{\mathbf{0}}X_0^\alpha, \mathbf{P}_{\mathbf{m}}^{\bm\varepsilon}X_\varepsilon^\alpha)}\leq $$
$$C[\beta(\varepsilon)+[\tau(\varepsilon)|\log(\tau(\varepsilon))|+\rho(\varepsilon)]^\theta]e^{-[(4+(\kappa+2)\theta )L_F(\lambda_m^\varepsilon)^\alpha+(\theta+1)\lambda_m^\varepsilon+3\theta] t}\,\,+$$
$$+\frac{\|ED\Psi_0^\eps-D\Psi_\varepsilon\mathbf{P}_{\mathbf{m}}^{\bm\varepsilon}E\|_\infty}{2}e^{-[4L_F(\lambda_m^\varepsilon)^\alpha+\lambda_m^\varepsilon] t},$$
where $C$ is a constant independent of $\eps$,  $0<\theta\leq\theta_F$ and $\theta<\theta_0$, and $\kappa$ is given by \eqref{cotaextensionproyeccion}. 

\end{lem}

\begin{re}
We denote by $\|ED\Psi_0^\eps-D\Psi_\eps\mathbf{P}_{\mathbf{m}}^{\bm\varepsilon} E\|_\infty$ the sup norm, that is 
\begin{equation}\label{supnorm}
\|ED\Psi_0^\eps-D\Psi_\eps \mathbf{P}_{\mathbf{m}}^{\bm\varepsilon}E\|_\infty=\sup_{p\in\mathbf{P}_{\mathbf{m}}^{\mathbf{0}}X_0^\alpha}\|ED\Psi_0^\eps(p)-D\Psi_\eps(\mathbf{P}_{\mathbf{m}}^{\bm\eps}Ep)\mathbf{P}_{\mathbf{m}}^{\bm\eps}E\|_{\mathcal{L}(\mathbf{P}_{\mathbf{m}}^{\mathbf{0}}X_0^\alpha, X_\eps^\alpha)},
\end{equation}
\end{re}

\begin{proof}  With the Variation of Constants Formula applied to \eqref{equationTheta*-section6} and \eqref{equationTheta-section6}, and denoting by $\Theta^\eps_0(t)=\Theta^\eps_0(j_0^{-1}(z),t)$ and  $\Theta_\eps(t)=\Theta_\eps(j_\eps^{-1}(z),t)$, we get

$$E\Theta^\eps_0(t)-\Theta_\varepsilon (t) \mathbf{P}_{\mathbf{m}}^{\bm\varepsilon}E =Ee^{-A_0\mathbf{P}_{\mathbf{m}}^{\mathbf{0}} t}-e^{-A_\varepsilon\mathbf{P}_{\mathbf{m}}^{\bm\varepsilon} t} \mathbf{P}_{\mathbf{m}}^{\bm\varepsilon}E+$$

$$+\int_t^0\left(Ee^{-A_0\mathbf{P}_{\mathbf{m}}^{\mathbf{0}} (t-s)}\mathbf{P}_{\mathbf{m}}^{\mathbf{0}}DF^\eps_0(u^\eps_0(s))(I+D\Psi_0^\eps(p^\eps_0(s)))\Theta^\eps_0(s)-\right. \qquad\qquad\qquad\qquad$$
$$\qquad\qquad\qquad \left. e^{-A_\varepsilon\mathbf{P}_{\mathbf{m}}^{\bm\varepsilon}(t-s)}\mathbf{P}_{\mathbf{m}}^{\bm\varepsilon}DF_\varepsilon(u_\varepsilon(s))(I+D\Psi_\varepsilon(p_\varepsilon(s)))\Theta_\varepsilon(s)\mathbf{P}_{\mathbf{m}}^{\bm\varepsilon} E \right) ds:= I'+\int_t^0I$$

We estimate now $I'$ and $I$. Notice first that $\|I'\|_{\mathcal{L}(\mathbf{P}_{\mathbf{m}}^{\mathbf{0}}X_0^\alpha, \mathbf{P}_{\mathbf{m}}^{\bm\varepsilon}X_\varepsilon^\alpha
)}$ is analyzed with Lemma 5.1, from \cite{Arrieta-Santamaria-DCDS} obtaining,
$$\|I'\|_{\mathcal{L}(\mathbf{P}_{\mathbf{m}}^{\mathbf{0}}X_0^\alpha, \mathbf{P}_{\mathbf{m}}^{\bm\varepsilon}X_\varepsilon^\alpha
)}\leq C_4e^{-(\lambda_m^0+1)t}\tau(\eps)$$

\par\medskip 

Moreover, for $I$ we get, the following decomposition:

$$I=Ee^{-A_0\mathbf{P}_{\mathbf{m}}^{\mathbf{0}} (t-s)}\mathbf{P}_{\mathbf{m}}^{\mathbf{0}}DF^\eps_0(u^\eps_0(s))(I+D\Psi_0^\eps(p^\eps_0(s)))\Theta^\eps_0(s)-$$
$$e^{-A_\varepsilon\mathbf{P}_{\mathbf{m}}^{\bm\varepsilon}(t-s)}\mathbf{P}_{\mathbf{m}}^{\bm\varepsilon}DF_\varepsilon(u_\varepsilon(s))(I+D\Psi_\varepsilon(p_\varepsilon(s)))\Theta_\varepsilon(s)\mathbf{P}_{\mathbf{m}}^{\bm\varepsilon} E=$$
$$=\left( Ee^{-A_0\mathbf{P}_{\mathbf{m}}^{\mathbf{0}} (t-s)}\mathbf{P}_{\mathbf{m}}^{\mathbf{0}}- e^{-A_\varepsilon\mathbf{P}_{\mathbf{m}}^{\bm\varepsilon}(t-s)}\mathbf{P}_{\mathbf{m}}^{\bm\varepsilon}E\right) DF^\eps_0(u^\eps_0(s))(I+D\Psi_0^\eps(p^\eps_0(s)))\Theta^\eps_0(s)$$

$$+e^{-A_\varepsilon\mathbf{P}_{\mathbf{m}}^{\bm\varepsilon}(t-s)}\mathbf{P}_{\mathbf{m}}^{\bm\varepsilon}\Big( EDF^\eps_0(u^\eps_0(s))- DF_\varepsilon(Eu^\eps_0(s))E\Big)(I+D\Psi_0^\eps(p^\eps_0(s)))\Theta^\eps_0(s)$$

$$+e^{-A_\varepsilon\mathbf{P}_{\mathbf{m}}^{\bm\varepsilon}(t-s)}\mathbf{P}_{\mathbf{m}}^{\bm\varepsilon}\Big( DF_\eps(Eu^\eps_0(s))- DF_\varepsilon(u_\eps (s))\Big)E(I+D\Psi_0^\eps(p^\eps_0(s)))\Theta^\eps_0(s)$$

$$+e^{-A_\varepsilon\mathbf{P}_{\mathbf{m}}^{\bm\varepsilon}(t-s)}\mathbf{P}_{\mathbf{m}}^{\bm\varepsilon} DF_\varepsilon(u_\varepsilon(s))
\Big( E(I+ D\Psi_0^\eps(p^\eps_0(s)))-(I+ D\Psi_\eps( \mathbf{P}_{\mathbf{m}}^{\bm\eps} E p^\eps_0(s)))E\Big)\Theta^\eps_0(s)$$

$$+e^{-A_\varepsilon\mathbf{P}_{\mathbf{m}}^{\bm\varepsilon}(t-s)}\mathbf{P}_{\mathbf{m}}^{\bm\varepsilon} DF_\varepsilon(u_\varepsilon(s))\Big(
(I+D\Psi_\eps( \mathbf{P}_{\mathbf{m}}^{\bm\eps} E p^\eps_0(s)))- (I+D\Psi_\varepsilon(p_\varepsilon(s)))\Big)E\Theta^\eps_0(s)$$

$$+e^{-A_\varepsilon\mathbf{P}_{\mathbf{m}}^{\bm\varepsilon}(t-s)}\mathbf{P}_{\mathbf{m}}^{\bm\varepsilon}DF_\varepsilon(u_\varepsilon(s))(I+D\Psi_\eps(p_\eps(s)))
\Big( E\Theta^\eps_0(s)- \Theta_\varepsilon(s)\mathbf{P}_{\mathbf{m}}^{\bm\varepsilon} E\Big)$$

$$=I_1+I_2+I_3+I_4+I_5+I_6.$$
Now we can study the norm $\|I\|_{\mathcal{L}(\mathbf{P}_{\mathbf{m}}^{\mathbf{0}}X_0^\alpha, \mathbf{P}_{\mathbf{m}}^{\bm\varepsilon}X_\varepsilon^\alpha)}$ analyzing the norm of each term separately.

\par\medskip 

By   Lemma \ref{Jnorm}, Lemma 5.1 from \cite{Arrieta-Santamaria-DCDS} and  \eqref{semigrupoproyectadoP-alpha} we have,

$$\|I_1\|_{\mathcal{L}(\mathbf{P}_{\mathbf{m}}^{\mathbf{0}}X_0^\alpha, \mathbf{P}_{\mathbf{m}}^{\bm\varepsilon}X_\varepsilon^\alpha)}\leq 2L_FC_4\tau(\varepsilon)e^{-(\lambda_m^0+1)t}e^{(-2L_F(\lambda_m^0)^\alpha+1) s}.$$
With the definition of $\beta(\eps)$ from \eqref{convergenceDF} and again Lemma \ref{Jnorm} and 
\eqref{semigrupoproyectadoP-alpha} $$\|I_2\|_{\mathcal{L}(\mathbf{P}_{\mathbf{m}}^{\mathbf{0}}X_0^\alpha, \mathbf{P}_{\mathbf{m}}^{\bm\varepsilon}X_\varepsilon^\alpha)}\leq 2(\lambda_m^\varepsilon)^\alpha\beta(\varepsilon)e^{-\lambda_m^\varepsilon t}e^{-2L_F(\lambda_m^\varepsilon)^\alpha s}.$$
To study the term $I_3$, again, from \eqref{convergenceDF}, \eqref{semigrupoproyectadoP-alpha},  Lemma \ref{Jnorm} and the properties on the norm of extension operator, see \eqref{cotaextensionproyeccion}, for $0<\theta\leq \theta_F$,
$$\|I_3\|_{\mathcal{L}(\mathbf{P}_{\mathbf{m}}^{\mathbf{0}}X_0^\alpha, \mathbf{P}_{\mathbf{m}}^{\bm\varepsilon}X_\varepsilon^\alpha)}\leq 2\kappa(\lambda_m^\varepsilon)^\alpha L\|Eu^\eps_0(s)-u_\varepsilon(s)\|_{X_\varepsilon^\alpha}^\theta e^{-\lambda_m^\varepsilon t}e^{-2L_F(\lambda_m^\varepsilon)^\alpha s}.$$
Remember that, 
$$u^\eps_0(s)=p^\eps_0(s)+\Psi^\eps_0(p^\eps_0(s))=p^\eps_0(s)+\Phi^\eps_0(j_0(p^\eps_0(s))),$$
and for $0<\eps\leq\eps_0$,
$$u_\eps(s)=p_\eps(s)+\Psi_\eps(p_\eps(s))=p_\eps(s)+\Phi_\eps(j_\eps(p_\eps(s))).$$
Then,
$$\|Eu^\eps_0(s)-u_\varepsilon(s)\|_{X_\varepsilon^\alpha}\leq$$
$$\|p_\eps(s)-Ep^\eps_0(s)\|_{X_\eps^\alpha}+ \|\Phi_\eps(j_\eps(p_\eps(s)))-\Phi_\eps(j_0(p^\eps_0(s))\|_{X_\eps^\alpha}+\|\Phi_\eps(j_0(p^\eps_0(s))-\Phi_0^\eps(j_0(p^\eps_0(s)))\|_{X_\eps^\alpha}\leq$$
$$\|p_\eps(s)-Ep^\eps_0(s)\|_{X_\eps^\alpha}+ |j_\eps(p_\eps(s))-j_0(p^\eps_0(s))|_{0,\alpha}+\|\Phi_\eps-E\Phi_0^\eps\|_{L^\infty(\mathbb{R}^m, X_\eps^\alpha)}.$$

Applying now Lemma 5.4 from \cite{Arrieta-Santamaria-DCDS}, we get
$$|j_\eps(p_\eps(s))-j_0(p^\eps_0(s))|_{0,\alpha}\leq (\kappa +1)\|p_\eps(s)-Ep^\eps_0(s)\|_{X_\eps^\alpha}+ (\kappa +1)C_P\tau(\eps)\|p_0^\eps\|_{X_0}$$

Applying also Theorem \ref{distaciavariedadesinerciales}, we get
 $$\|\Phi_\eps-E\Phi_0^\eps\|_{L^\infty(\mathbb{R}^m, X_\eps^\alpha)}\leq C[\tau(\varepsilon)|\log(\tau(\varepsilon))|+\rho(\varepsilon)]$$

Hence, 
$$\|Eu^\eps_0(s)-u_\varepsilon(s)\|_{X_\varepsilon^\alpha}\leq$$
$$ (\kappa+2)\|p_\eps(s)-Ep^\eps_0(s)\|_{X_\eps^\alpha}+ (\kappa +1)C_P\tau(\eps)\|p_0^\eps\|_{X_0}+C[\tau(\varepsilon)|\log(\tau(\varepsilon))|+\rho(\varepsilon)].$$

To estimate now $\|p_\eps(s)-Ep^\eps_0(s)\|_{X_\eps^\alpha}$ we follow Lemma 5.6 from \cite{Arrieta-Santamaria-DCDS} and to estimate $\|p_0^\eps\|_{X_0}$ we use Lemma 5.5 from \cite{Arrieta-Santamaria-DCDS} also. 

Putting all these estimates together, we get 

$$\|Eu^\eps_0(s)-u_\varepsilon(s)\|_{X_\varepsilon^\alpha}\leq$$
$$\leq (\kappa+2)\left( \frac{L_F}{(\lambda_m^\eps)^{1-\alpha}}\tau(\varepsilon)|\log(\tau(\varepsilon))|+\rho(\varepsilon)+K_2e^{-2s}\tau(\varepsilon)\right) e^{-[(\kappa+2)L_F(\lambda_m^\varepsilon)^\alpha+\lambda_m^\varepsilon]s}+$$
$$+ (\kappa+1)C_P\tau(\eps)(R+C_F)e^{-\lambda_m^\varepsilon s\theta}+C[\tau(\varepsilon)|\log(\tau(\varepsilon))|+\rho(\varepsilon)]\leq $$
$$\leq C [\tau(\varepsilon)|\log(\tau(\varepsilon))|+\rho(\varepsilon)] e^{-[(\kappa+2)L_F(\lambda_m^\eps)^\alpha+\lambda_m^\eps+3]s},$$
with $C>0$ independent of $\varepsilon$.  Observe that since $s\leq 0$, we have $e^{-[(\kappa+2)L_F(\lambda_m^\eps)^\alpha+\lambda_m^\eps+3]s}\geq 1$.

%
%
\bigskip

Hence,
$$\resizebox{.99\hsize}{!}{$\|I_3\|_{\mathcal{L}(\mathbf{P}_{\mathbf{m}}^{\mathbf{0}}X_0^\alpha, \mathbf{P}_{\mathbf{m}}^{\bm\varepsilon}X_\varepsilon^\alpha)}\leq 2\kappa(\lambda_m^\varepsilon)^\alpha LC [\tau(\varepsilon)|\log(\tau(\varepsilon))|\mathord+\rho(\varepsilon)]^\theta e^{-\lambda_m^\varepsilon t}e^{-[(2\mathord+(\kappa\mathord+2)\theta)L_F(\lambda_m^\varepsilon)^\alpha\mathord+\theta\lambda_m^\eps+3\theta] s}$}.$$
By Lemma \ref{Jnorm}, 
we have,
$$\|I_4\|_{\mathcal{L}(\mathbf{P}_{\mathbf{m}}^{\mathbf{0}}X_0^\alpha, \mathbf{P}_{\mathbf{m}}^{\bm\varepsilon}X_\varepsilon^\alpha)}\leq (\lambda_m^\varepsilon)^\alpha L_F\|ED\Psi_0^\eps-D\Psi_\eps \mathbf{P}_{\mathbf{m}}^{\bm\varepsilon}E\|_\infty e^{-\lambda_m^\varepsilon t}e^{-2L_F(\lambda_m^\varepsilon)^\alpha s}.$$

By Section \ref{smoothness-subsection}, $D\Psi_\eps\in \mathcal{E}_\eps^{\theta,M}$ for $0<\theta\leq\theta_F$ and $\theta<\theta_0$. Applying estimate (\ref{cotaextensionproyeccion}), Lemma \ref{Jnorm} and Lemma 5.6 from \cite{Arrieta-Santamaria-DCDS}, we have,
$$\resizebox{14.3cm}{!}{$\|I_5\|_{\mathcal{L}(\mathbf{P}_{\mathbf{m}}^{\mathbf{0}}X_0^\alpha, \mathbf{P}_{\mathbf{m}}^{\bm\varepsilon}X_\varepsilon^\alpha)}\leq \kappa L_F(\lambda_m^\eps)^\alpha M(\tau(\eps)|\log(\tau(\eps))|\mathord+\rho(\eps))^\theta e^{-\lambda_m^\eps t}e^{-[(2\mathord+(\kappa\mathord+2)\theta)L_F(\lambda_m^\eps)^\alpha\mathord+\theta\lambda_m^\eps\mathord+3\theta] s} $}$$
Finally, the norm of term $I_6$ is estimated by,
$$\|I_6\|_{\mathcal{L}(\mathbf{P}_{\mathbf{m}}^{\mathbf{0}}X_0^\alpha, \mathbf{P}_{\mathbf{m}}^{\bm\varepsilon}X_\varepsilon^\alpha)}\leq 2(\lambda_m^\varepsilon)^\alpha L_F e^{-\lambda_m^\eps (t-s)}\|E\Theta^\eps_0(s)-\Theta_\eps(s)\mathbf{P}_{\mathbf{m}}^{\bm\eps}E \|_{\mathcal{L}(\mathbf{P}_{\mathbf{m}}^{\mathbf{0}}X_0^\alpha, \mathbf{P}_{\mathbf{m}}^{\bm\eps}X_\eps^\alpha)}.$$
Putting all together,
$$\|I\|_{\mathcal{L}(\mathbf{P}_{\mathbf{m}}^{\mathbf{0}}X_0^\alpha, \mathbf{P}_{\mathbf{m}}^{\bm\varepsilon}X_\varepsilon^\alpha)}\leq$$
$$ CL_FL(\lambda_m^\eps)^\alpha\left[\beta(\eps)\mathord+(\tau(\eps)|\log(\tau(\eps))|\mathord+\rho(\eps))^\theta\right]e^{-\lambda_m^\eps t} e^{-[(2\mathord+(\kappa\mathord+2)\theta)L_F(\lambda_m^\eps)^\alpha\mathord+\theta\lambda_m^\eps\mathord+3\theta]s}$$
$$+(\lambda_m^\eps)^\alpha L_F\|ED\Psi_0^\eps-D\Psi_\eps\mathbf{P}_{\mathbf{m}}^{\bm\varepsilon} E\|_\infty e^{-\lambda_m^\varepsilon t}e^{-2L_F(\lambda_m^\varepsilon)^\alpha s}+$$
$$+ 2(\lambda_m^\eps)^\alpha L_F e^{-\lambda_m^\eps(t-s)}\|E\Theta^\eps_0(s)-\Theta_\eps(s)\mathbf{P}_{\mathbf{m}}^{\bm\eps}E \|_{\mathcal{L}(\mathbf{P}_{\mathbf{m}}^{\mathbf{0}}X_0^\alpha, \mathbf{P}_{\mathbf{m}}^{\bm\eps}X_\eps^\alpha)}.$$ 

Then,
$$\|E\Theta^\eps_0(t)-\Theta_\eps(t)\mathbf{P}_{\mathbf{m}}^{\bm\eps}E \|_{\mathcal{L}(\mathbf{P}_{\mathbf{m}}^{\mathbf{0}}X_0^\alpha, \mathbf{P}_{\mathbf{m}}^{\bm\eps}X_\eps^\alpha)}\leq \|I'\|_{\mathcal{L}(\mathbf{P}_{\mathbf{m}}^{\mathbf{0}}X_0^\alpha, \mathbf{P}_{\mathbf{m}}^{\bm\eps}X_\eps^\alpha)}+\int_t^0\|I\|_{\mathcal{L}(\mathbf{P}_{\mathbf{m}}^{\mathbf{0}}X_0^\alpha, \mathbf{P}_{\mathbf{m}}^{\bm\varepsilon}X_\varepsilon^\alpha)}\leq $$
$$\leq C_4 e^{-(\lambda_m^0+1)t}\tau(\eps)+$$
$$CL_FL(\lambda_m^\eps)^\alpha\left[\beta(\eps)\mathord+(\tau(\eps)|\log(\tau(\eps))|\mathord+\rho(\eps))^\theta\right]e^{-\lambda_m^\eps t}\int_t^0  e^{-[(2\mathord+(\kappa\mathord+2)\theta)L_F(\lambda_m^\eps)^\alpha\mathord+\theta\lambda_m^\eps\mathord+3\theta] s}ds$$
$$+(\lambda_m^\eps)^\alpha L_F\|ED\Psi_0^\eps-D\Psi_\eps\mathbf{P}_{\mathbf{m}}^{\bm\varepsilon} E\|_\infty e^{-\lambda_m^\varepsilon t}\int_t^0 e^{-2L_F(\lambda_m^\varepsilon)^\alpha s}ds+$$
$$+ 2(\lambda_m^\eps)^\alpha L_F e^{-\lambda_m^\eps t}\int_t^0 e^{\lambda_m^\eps s}\|E\Theta^\eps_0(s)-\Theta_\eps(s)\mathbf{P}_{\mathbf{m}}^{\bm\eps}E \|_{\mathcal{L}(\mathbf{P}_{\mathbf{m}}^{\mathbf{0}}X_0^\alpha, \mathbf{P}_{\mathbf{m}}^{\bm\eps}X_\eps^\alpha)}ds.$$
So, we have,
$$\|E\Theta^\eps_0(t)-\Theta_\eps(t)\mathbf{P}_{\mathbf{m}}^{\bm\eps}E \|_{\mathcal{L}(\mathbf{P}_{\mathbf{m}}^{\mathbf{0}}X_0^\alpha, \mathbf{P}_{\mathbf{m}}^{\bm\eps}X_\eps^\alpha)}\leq $$
$$\leq  C \left[\beta(\eps)+(\tau(\eps)|\log(\tau(\eps))|+\rho(\eps))^\theta\right]e^{-[(2+(\kappa+2)\theta)L_F(\lambda_m^\eps)^\alpha +(\theta+1)\lambda_m^\eps+3\theta] t}+$$
$$+\frac{\|ED\Psi_0^\eps-D\Psi_\eps\mathbf{P}_{\mathbf{m}}^{\bm\varepsilon} E\|_\infty}{2}e^{-[2L_F(\lambda_m^\eps)^\alpha +\lambda_m^\eps] t}+ $$
$$+ 2(\lambda_m^\eps)^\alpha L_F e^{-\lambda_m^\eps t}\int_t^0 e^{\lambda_m^\eps s}\|E\Theta^\eps_0(s)-\Theta_\eps(s)\mathbf{P}_{\mathbf{m}}^{\bm\eps}E \|_{\mathcal{L}(\mathbf{P}_{\mathbf{m}}^{\mathbf{0}}X_0^\alpha, \mathbf{P}_{\mathbf{m}}^{\bm\eps}X_\eps^\alpha)}ds.$$
Applying Gronwall inequality,
$$\|E\Theta^\eps_0(t)-\Theta_\eps(t)\mathbf{P}_{\mathbf{m}}^{\bm\eps}E \|_{\mathcal{L}(\mathbf{P}_{\mathbf{m}}^{\mathbf{0}}X_0^\alpha, \mathbf{P}_{\mathbf{m}}^{\bm\eps}X_\eps^\alpha)}\leq$$
$$\leq C \left[\beta(\eps)+(\tau(\eps)|\log(\tau(\eps))|+\rho(\eps))^\theta\right]e^{-[(4+(\kappa+2)\theta)L_F(\lambda_m^\eps)^\alpha +(\theta+1)\lambda_m^\eps+3\theta] t}+$$
$$+\frac{\|ED\Psi_0^\eps-D\Psi_\eps \mathbf{P}_{\mathbf{m}}^{\bm\varepsilon}E\|_\infty}{2}e^{-[4L_F(\lambda_m^\eps)^\alpha+\lambda_m^\eps] t},$$
with $C>0$ a constant independent of $\eps$ and $0<\theta\leq\theta_F$ with $\theta<\theta_0$. \end{proof}


\bigskip
We show now the convergence of the differential of inertial manifolds and establish a rate for this convergence. For this, we define $\theta_1$ and $\tilde{\theta}$ as follows, 
\begin{equation}\label{theta1}
\theta_1= \frac{\lambda_{m+1}^0-\lambda_m^0-4L_F(\lambda_m^0)^\alpha}{(\kappa+2)L_F(\lambda_m^0)^\alpha+\lambda_m^0+3},
\end{equation}
and,
\begin{equation}\label{theta*}
\tilde{\theta}=\min\left\{\theta_F,\, \theta_0,\, \theta_1\right\}.
\end{equation}

\begin{prop}\label{differential-convergence}
With $\Phi_0^\eps$ and $\Phi_\eps$ the inertial manifolds,  and if $\theta<\tilde{\theta}$ , 
we have the following estimate 
\begin{equation}\label{C1-convergence}
\|ED\Phi_0^\eps\mathord-D\Phi_\varepsilon\|_{C^1(\mathbb{R}^m, X_\varepsilon^\alpha)}\leq  C\left[\beta(\varepsilon)\mathord+\Big(\tau(\varepsilon)|\log(\tau(\varepsilon))|\mathord+\rho(\varepsilon)\Big)^\theta\right]
\end{equation}
where $C$ is a constant independent of $\eps$. 
\end{prop}

\begin{proof}  Taking into account the estimate obtained in Theorem \ref{distaciavariedadesinerciales}, it remains to estimate $\|ED\Phi_0^\eps-D\Phi_\eps\|_{L^\infty(\mathbb{R}^m,\, \mathcal{L}(\mathbb{R}^m, X_\eps^\alpha))}$, that is, 
$$\sup_{z\in\mathbb{R}^m}\|ED\Phi_0^\eps(z)\mathord-D\Phi_\varepsilon(z)\|_{\mathcal{L}(\mathbb{R}^m, X_\varepsilon^\alpha)}.$$

But we know that, 
$$\sup_{z\in\mathbb{R}^m}\|ED\Phi_0^\eps(z)-D\Phi_\varepsilon(z)\|_{\mathcal{L}(\mathbb{R}^m, X_\varepsilon^\alpha)}=$$
$$=\sup_{z\in\mathbb{R}^m}\|ED\Psi_0^\eps(j_0^{-1}(z))j_0^{-1}-D\Psi_\varepsilon(j_\varepsilon^{-1}(z))\mathbf{P}_{\mathbf{m}}^{\bm\varepsilon}Ej_0^{-1}\|_{\mathcal{L}(\mathbb{R}^m, X_\varepsilon^\alpha)}=$$
$$=\sup_{z\in\mathbb{R}^m}\|ED\Psi_0^\eps(j_0^{-1}(z))-D\Psi_\varepsilon(\mathbf{P}_{\mathbf{m}}^{\bm\varepsilon}Ej_0^{-1}(z))\mathbf{P}_{\mathbf{m}}^{\bm\varepsilon}E\|_{\mathcal{L}(\mathbf{P}_{\mathbf{m}}^{\mathbf{0}}X_0^\alpha, X_\varepsilon^\alpha)}=$$
$$=\sup_{p^\eps_0\in\mathbf{P}_{\mathbf{m}}^{\mathbf{0}}X_0^\alpha}\|ED\Psi_0^\eps(p^\eps_0)-D\Psi_\varepsilon(\mathbf{P}_{\mathbf{m}}^{\bm\varepsilon}Ep^\eps_0)\mathbf{P}_{\mathbf{m}}^{\bm\varepsilon}E\|_{\mathcal{L}(\mathbf{P}_{\mathbf{m}}^{\mathbf{0}}X_0^\alpha, X_\varepsilon^\alpha)}=\|ED\Psi_0^\eps-D\Psi_\varepsilon E\|_\infty.$$

We have applied $|j_0(p^\eps_0)|_{0,\alpha}=\|p^\eps_0\|_{X_0^\alpha}$ for any $p^\eps_0\in \mathbf{P}_{\mathbf{m}}^{\mathbf{0}}X_0$, see (\ref{normajepsilon}).
\bigskip

Then, for $z'\in\mathbb{R}^m$, with the definition (\ref{differential-Psi}), and denoting again by $\Theta^\eps_0(t)=\Theta^\eps_0(j_0^{-1}(z),t)$ and  $\Theta_\eps(t)=\Theta_\eps(j_\eps^{-1}(z),t)$,  we have
$$ED\Psi_0^\eps(j_0^{-1}(z))j_0^{-1}(z')-D\Psi_\varepsilon(\mathbf{P_m^\eps}E\circ j_0^{-1}(z))\mathbf{P_m^\eps}E\circ j_0^{-1}(z')=$$
$$=\int_{-\infty}^0\left( Ee^{A_0\mathbf{Q}_{\mathbf{m}}^{\mathbf{0}}s}\mathbf{Q}_{\mathbf{m}}^{\mathbf{0}}DF^\eps_0(u^\eps_0(s))(I+D\Psi_0^\eps(p^\eps_0(s)))\Theta^\eps_0(s)j_0^{-1}(z')\right.$$

$$\left. -e^{A_\varepsilon\mathbf{Q}_{\mathbf{m}}^{\bm\varepsilon}s}\mathbf{Q}_{\mathbf{m}}^{\bm\varepsilon}DF_\varepsilon(u_\varepsilon(s))(I+D\Psi_\varepsilon(p_\varepsilon(s)))\Theta_\varepsilon(s)\mathbf{P}_{\mathbf{m}}^{\bm\varepsilon}Ej_0^{-1}(z')ds\right)=\int_{-\infty}^0I$$

But, the integrand $I$ can be decomposed, in a similar way as above in the proof of Lemma \ref{Jdistance}, as

$$I=\left(Ee^{A_0\mathbf{Q}_{\mathbf{m}}^{\mathbf{0}}s}\mathbf{Q}_{\mathbf{m}}^{\mathbf{0}}-e^{A_\varepsilon\mathbf{Q}_{\mathbf{m}}^{\bm\varepsilon} s}\mathbf{Q}_{\mathbf{m}}^{\bm\varepsilon} E\right)DF^\eps_0(u^\eps_0(s))(I+D\Psi_0^\eps(p^\eps_0(s)))\Theta^\eps_0(s)j_0^{-1}(z')+$$

$$+e^{A_\varepsilon\mathbf{Q}_{\mathbf{m}}^{\bm\varepsilon}s}\mathbf{Q}_{\mathbf{m}}^{\bm\varepsilon}\Big(EDF^\eps_0(u^\eps_0)-DF_\varepsilon(Eu^\eps_0(s))E\Big)(I+D\Psi_0^\eps(p^\eps_0(s)))\Theta^\eps_0(s)j_0^{-1}(z')$$

$$+e^{A_\varepsilon\mathbf{Q}_{\mathbf{m}}^{\bm\varepsilon}s}\mathbf{Q}_{\mathbf{m}}^{\bm\varepsilon}\Big(DF_\eps(Eu^\eps_0)-DF_\varepsilon(u_\varepsilon(s))\Big)E(I+D\Psi_0^\eps(p^\eps_0(s)))\Theta^\eps_0(s)j_0^{-1}(z')$$

$$+e^{A_\varepsilon\mathbf{Q}_{\mathbf{m}}^{\bm\varepsilon}s}\mathbf{Q}_{\mathbf{m}}^{\bm\varepsilon}DF_\varepsilon(u_\varepsilon(s))\Big( E(I\mathord+D\Psi_0^\eps(p^\eps_0(s)))\mathord-(I\mathord+D\Psi_\eps(\mathbf{P}_{\mathbf{m}}^{\bm\eps}Ep^\eps_0(s)))E\Big)\Theta^\eps_0(s)j_0^{-1}(z')
$$

$$+e^{A_\varepsilon\mathbf{Q}_{\mathbf{m}}^{\bm\varepsilon}s}\mathbf{Q}_{\mathbf{m}}^{\bm\varepsilon}DF_\varepsilon(u_\varepsilon(s))\Big((I\mathord+D\Psi_\eps(\mathbf{P}_{\mathbf{m}}^{\bm\eps}Ep^\eps_0(s))))\mathord-(I\mathord+D\Psi_\eps(p_\eps(s)))\Big)E\Theta^\eps_0(s)j_0^{-1}(z')
$$

$$+e^{A_\varepsilon\mathbf{Q}_{\mathbf{m}}^{\bm\varepsilon}s}\mathbf{Q}_{\mathbf{m}}^{\bm\varepsilon}DF_\varepsilon(u_\varepsilon(s))(I+D\Psi_\eps(p_\eps(s)))\Big( E\Theta^\eps_0(s)-\Theta_\varepsilon(s)\mathbf{P}_{\mathbf{m}}^{\bm\varepsilon}E\Big)j_0^{-1}(z')$$
$$=I_1+I_2+I_3+I_4+I_5+I_6.$$

Applying Lemma 5.3 from \cite{Arrieta-Santamaria-DCDS} and Lemma \ref{Jnorm},

$$\|I_1\|_{X_\eps^\alpha}\leq 2 C_5 L_F l_\eps^\alpha(-s)e^{[-2L_F(\lambda_m^\eps)^\alpha+\lambda_{m+1}^\eps-\lambda_m^\eps-1]s}|z'|_{0,\alpha}.$$
Following the same steps as in the proof of Lemma \ref{Jdistance}, we obtain,
$$\|I_2\|_{X_\eps^\alpha}\leq 2(\lambda_{m+1}^\eps)^\alpha\beta(\eps)e^{[-2L_F(\lambda_m^\eps)^\alpha+\lambda_{m+1}^\eps-\lambda_m^\eps]s}|z'|_{0,\alpha}, $$
$$\|I_4\|_{X_\eps^\alpha}\leq (\lambda_{m+1}^\eps)^\alpha L_F\|ED\Psi_0^\eps-D\Psi_\eps E\|_\infty e^{[-2L_F(\lambda_m^\eps)^\alpha+\lambda_{m+1}^\eps-\lambda_m^\eps]s}|z'|_{0,\alpha}. $$
For the sake of clarity we will denote by
\begin{equation}\label{def-exponents-2}
\begin{array}{l}
\Lambda_3=-(2+(\kappa+2)\theta)L_F(\lambda_m^\eps)^\alpha+\lambda_{m+1}^\eps-(\theta+1)\lambda_m^\eps-3\theta \\
\Lambda_4=-(4+(\kappa+2)\theta)L_F(\lambda_m^\eps)^\alpha+\lambda_{m+1}^\eps-(\theta+1)\lambda_m^\eps-3\theta.
\end{array}
\end{equation}

Then, we have,
$$\|I_3\|_{X_\eps^\alpha}\leq 2\kappa(\lambda_{m+1}^\eps)^\alpha LC[\tau(\eps)|\log(\tau(\eps))|+\rho(\eps)]^\theta e^{\Lambda_3s}|z'|_{0,\alpha}, $$

$$\|I_5\|_{X_\eps^\alpha}\leq \kappa L_F(\lambda_{m+1}^\eps)^\alpha M C\left(\tau(\eps)|\log(\tau(\eps))|+\rho(\eps)\right)^\theta e^{\Lambda_3 s}|z'|_{0,\alpha}, $$
and for the norm of $I_6$ we apply Lemma \ref{Jdistance},
$$\|I_6\|_{X_\eps^\alpha}\leq\left( 2(\lambda_{m+1}^\eps)^\alpha L_F C \left[\beta(\eps)+(\tau(\eps)|\log(\tau(\eps))|+\rho(\eps))^\theta\right]e^{\Lambda_4 s}+\right.$$
$$\left.(\lambda_{m+1}^\eps)^\alpha L_F\|ED\Psi_0-D\Psi_\eps E\|_\infty e^{[-4L_F(\lambda_m^\eps)^\alpha+\lambda_{m+1}^\eps-\lambda_m^\eps] s}\right)|z'|_{0,\alpha}.$$
Putting everything together, $\|I\|_{X_\eps^\alpha}\leq \|I_1\|_{X_\eps^\alpha}+\|I_2\|_{X_\eps^\alpha}+\|I_3\|_{X_\eps^\alpha}+\|I_4\|_{X_\eps^\alpha}+\|I_5\|_{X_\eps^\alpha}+\|I_6\|_{X_\eps^\alpha}$, so, 
$$\int_{-\infty}^0\|I\|_{X_\eps^\alpha}ds\leq 2 C_5 L_F|z'|_{0,\alpha} \int_{-\infty}^0l_\eps^\alpha(-s)e^{[-2L_F(\lambda_m^\eps)^\alpha+\lambda_{m+1}^\eps-\lambda_m^\eps-1]s}ds+$$
$$+ 2(\lambda_{m+1}^\eps)^\alpha\beta(\eps)|z'|_{0,\alpha}\int_{-\infty}^0 e^{[-2L_F(\lambda_m^\eps)^\alpha+\lambda_{m+1}^\eps-\lambda_m^\eps]s}ds+$$
$$+2\kappa(\lambda_{m+1}^\eps)^\alpha LC[\tau(\eps)|\log(\tau(\eps))|+\rho(\eps)]^\theta|z'|_{0,\alpha}\int_{-\infty}^0 e^{\Lambda_3s}ds+$$
$$+(\lambda_{m+1}^\eps)^\alpha L_F\|ED\Psi_0-D\Psi_\eps E\|_\infty|z'|_{0,\alpha}\int_{-\infty}^0 e^{[-2L_F(\lambda_m^\eps)^\alpha+\lambda_{m+1}^\eps-\lambda_m^\eps]s}ds+$$
$$+\kappa L_F(\lambda_{m+1}^\eps)^\alpha MC\left(\tau(\eps)|\log(\tau(\eps))|+\rho(\eps)\right)^\theta|z'|_{0,\alpha}\int_{-\infty}^0 e^{\Lambda_3 s}ds+$$
$$+2(\lambda_{m+1}^\eps)^\alpha L_F C \left[\beta(\eps)+(\tau(\eps)|\log(\tau(\eps))|+\rho(\eps))^\theta\right]|z'|_{0,\alpha}\int_{-\infty}^0e^{\Lambda_4 s}ds+$$
$$+(\lambda_{m+1}^\eps)^\alpha L_F\|ED\Psi_0-D\Psi_\eps E\|_\infty|z'|_{0,\alpha}\int_{-\infty}^0e^{[-4L_F(\lambda_m^\eps)^\alpha+\lambda_{m+1}^\eps-\lambda_m^\eps]s}ds.$$

By Lemma 3.10 from \cite{Arrieta-Santamaria-DCDS}, the gap conditions described in Proposition \ref{existenciavariedadinercial} and $0<\theta<\tilde{\theta}$, see (\ref{theta*}), for $\eps$ small enough,

$$\leq \left(C[\beta(\varepsilon)+(\tau(\varepsilon)|\log(\tau(\varepsilon))|+\rho(\varepsilon))^\theta]+\frac{1}{2}\|ED\Psi_0^\eps-D\Psi_\eps E\|_\infty\right)|z'|_{0,\alpha}$$

Hence,
$$\|[ED\Psi_0^\eps(j_0^{-1}(z))- D\Psi_\varepsilon(\mathbf{P}_{\mathbf{m}}^{\bm\varepsilon}Ej_0^{-1}(z))\mathbf{P}_{\mathbf{m}}^{\bm\varepsilon}E]j_0^{-1}(z')\|_{X_\varepsilon^\alpha}\leq$$
$$\leq  \left(C\left[\beta(\varepsilon)+(\tau(\varepsilon)|\log(\tau(\varepsilon))|+\rho(\varepsilon))^\theta\right]+\frac{1}{2}\|ED\Psi_0^\eps-D\Psi_\varepsilon E\|_\infty\right)|z'|_{0,\alpha}.$$
Since $\Psi_\varepsilon$ and $\Psi_0^\eps$ have bounded support, we consider the sup norm described in (\ref{supnorm}) for $u_0\in\mathbf{P}_{\mathbf{m}}^{\mathbf{0}}X_0^\alpha$ with $\|u_0\|_{X_0^\alpha}\leq 2\mathcal{R}$, with $\mathcal{R}>0$ an upper bound of the support of all $\Psi_\varepsilon$, $0<\varepsilon\leq\varepsilon_0$, and of $\Psi_0^\eps$.

So,
$$\|ED\Psi_0^\eps-D\Psi_\varepsilon E\|_\infty=\qquad\qquad\qquad
$$
$$=\sup_{p\in\mathbf{P}_{\mathbf{m}}^{\mathbf{0}}X_0^\alpha, \|p\|_{X_0^\alpha}\leq2\mathcal{R}}\|ED\Psi_0^\eps(p)-D\Psi_\varepsilon(\mathbf{P}_{\mathbf{m}}^{\bm\varepsilon}Ep)\mathbf{P}_{\mathbf{m}}^{\bm\varepsilon}E\|_{\mathcal{L}(\mathbf{P}_{\mathbf{m}}^{\mathbf{0}}X_0^\alpha, X_\varepsilon^\alpha)}$$
$$=\sup_{z\in\mathbb{R}^m, |z|_{0,\alpha}\leq 2\mathcal{R}}\|ED\Psi_0^\eps(j_0^{-1}(z))- D\Psi_\varepsilon(\mathbf{P}_{\mathbf{m}}^{\bm\varepsilon}Ej_0^{-1}(z))\mathbf{P}_{\mathbf{m}}^{\bm\varepsilon}E\|_{\mathcal{L}(\mathbf{P}_{\mathbf{m}}^{\mathbf{0}}X_0^\alpha, X_\varepsilon^\alpha)}\leq $$
$$\leq  C\left[\beta(\varepsilon)+(\tau(\varepsilon)|\log(\tau(\varepsilon))|+\rho(\varepsilon))^\theta\right]+\frac{1}{2}\|ED\Psi_0^\eps-D\Psi_\varepsilon E\|_\infty.$$
which implies,
$$\|ED\Psi_0^\eps-D\Psi_\varepsilon E\|_\infty \leq 2C\left[\beta(\varepsilon)+(\tau(\varepsilon)|\log(\tau(\varepsilon))|+\rho(\varepsilon))^\theta\right],$$
with $\theta<\tilde{\theta}$.

Hence, for $\theta<\tilde{\theta}$,
$$\sup_{z\in\mathbb{R}^m}\|ED\Phi_0^\eps(z)-D\Phi_\varepsilon(z)\|_{\mathcal{L}(\mathbb{R}^m, X_\varepsilon^\alpha)}\leq 2 C\left[\beta(\varepsilon)+(\tau(\varepsilon)|\log(\tau(\varepsilon))|+\rho(\varepsilon))^\theta\right].$$
Applying Theorem \ref{distaciavariedadesinerciales}, then
$$\|ED\Phi_0^\eps-D\Phi_\varepsilon\|_{C^1(\mathbb{R}^m, X_\varepsilon^\alpha)}\leq C\left[\beta(\varepsilon)+(\tau(\varepsilon)|\log(\tau(\varepsilon))|+\rho(\varepsilon))^\theta\right].$$

Which concludes the proof of the proposition. 
\end{proof}

\bigskip
With this estimate we can analyze in detail the $C^{1, \theta}$-convergence of inertial manifolds for some $\theta<\tilde{\theta}$, small enough. We introduce now the proof of the main result of this subsection.
\bigskip

\begin{proof} {\sl  (of Theorem \ref{convergence-C^1-theo})}   We want to show the existence of $\theta^*$ such that we can prove the convergence of the inertial manifolds $\Phi_\eps$ to $\Phi^\eps_0$, when $\eps$ tends to zero in the $C^{1, \theta}$ topology for $\theta<\theta^*$ and obtain a rate of this convergence. That is, an estimate of $\|\Phi_\eps- E\Phi_0^\eps\|_{C^{1, \theta}(\mathbb{R}^m, X_\eps^\alpha)}$. Let us choose $\theta^*<\tilde{\theta}$ as close as we want to $\tilde{\theta}$, where $\tilde{\theta}$ is given by \eqref{theta*}, so that Proposition \ref{differential-convergence} holds.
\bigskip

As we have mentioned,
$$\|\Phi_\eps- E\Phi_0^\eps\|_{C^{1, \theta}(\mathbb{R}^m, X_\eps^\alpha)}=\|\Phi_\eps- E\Phi_0^\eps\|_{C^1(\mathbb{R}^m, X_\eps^\alpha)}+$$
$$+ \sup_{z, z'\in\mathbb{R}^m}\frac{\|(D\Phi_\eps- ED\Phi_0^\eps)(z)-(D\Phi_\eps- ED\Phi_0^\eps)(z')\|_{\mathcal{L}(\mathbb{R}^m, X_\eps^\alpha)}}{|z-z'|_{\eps,\alpha}^{\theta}}=$$
$$=I_1+I_2.$$
For $\theta<\theta^*$, $I_2$ can be written as $I_2=I_{21}\cdot I_{22}$, where 
$$I_{21}=\left(\frac{\|(D\Phi_\eps- ED\Phi_0^\eps)(z)-(D\Phi_\eps- ED\Phi_0^\eps)(z')\|_{\mathcal{L}(\mathbb{R}^m, X^\alpha_\eps)}}{|z-z'|_{\eps,\alpha}^{\theta^*}}\right)^{\frac{\theta}{\theta^*}}$$ 
$$I_{22}=\|(D\Phi_\eps- ED\Phi_0^\eps)(z)-(D\Phi_\eps- ED\Phi_0^\eps)(z')\|^{1-\frac{\theta}{\theta^*}}_{\mathcal{L}(\mathbb{R}^m, X_\eps^\alpha)}$$

Note that, since for each $\varepsilon> 0$, $\Phi_\varepsilon=\Psi_\varepsilon\circ j_\varepsilon^{-1}$, and $\Phi^\eps_0=\Psi^\eps_0\circ j_0^{-1}$ then by the chain rule, for all $z, \bar v\in\mathbb{R}^m$,
$$D\Phi_\varepsilon(z)z'=D\Psi_\varepsilon(j_\varepsilon^{-1}(z))( j_\varepsilon^{-1}(z')),$$
$$D\Phi^\eps_0(z)z'=D\Psi^\eps_0(j_0^{-1}(z))( j_0^{-1}(z')).$$
Also, notice that from the definition of $j_\eps$, $j_0$, we have $j_\eps\circ \mathbf{P}_{\mathbf{m}}^{\bm\eps}E=j_0$ or equivalently $j_\eps^{-1}= \mathbf{P}_{\mathbf{m}}^{\bm\eps}E\circ j_0^{-1}$. 

Then, applying \eqref{des-normas} to the denominator,
$$I_{21}\leq
\resizebox{14.5cm}{!}{$\left(\frac{\|(D\Psi_\eps(j_\eps^{-1}(z))\mathord-D\Psi_\eps(j_\eps^{-1}(z')))j_\eps^{-1}\mathord+(ED\Psi_0^\eps(j_0^{-1}(z'))\mathord-ED\Psi_0^\eps(j_0^{-1}(z)))j_0^{-1}\|_{\mathcal{L}(\mathbb{R}^m, X_\eps^\alpha)}}{(1-\delta)^{\theta^*}\|j_0^{-1}(z)-j_0^{-1}(z')\|_{X_0^\alpha}^{\theta^*}}\right)^{\frac{\theta}{\theta^*}}$}$$
Since in the previous subsection we have proved $D\Psi_\eps\in \mathcal{E}_\eps^{\theta,M}$, with $\theta<\theta_0$, in particular we have $D\Psi_\eps\in \mathcal{E}_\eps^{\theta,M}$, with $\theta<\tilde{\theta}$. Without loss of generality we consider $D\Psi_\eps\in \mathcal{E}_\eps^{\theta^*,M}$. Moreover, $\|j_\eps^{-1}\|_{\mathcal{L}(\mathbb{R}^m, \mathbf{P}_{\mathbf{m}}^{\bm\eps} X_\eps^\alpha)}=\|\mathbf{P}_{\mathbf{m}}^{\bm\eps}E\circ j_0^{-1}\|_{\mathcal{L}(\mathbb{R}^m, \mathbf{P}_{\mathbf{m}}^{\bm\eps} X_\eps^\alpha)}\leq\kappa$, see (\ref{normajepsilon}) and (\ref{cotaextensionproyeccion}). Then, we obtain
$$I_{21}\leq \frac{(M\kappa(\kappa+1))^{\frac{\theta}{\theta^*}}\|j_0^{-1}(z)-j_0^{-1}(z')\|_{X_0^\alpha}^{\theta}}{(1-\delta)^\theta\|j_0^{-1}(z)-j_0^{-1}(z')\|_{X_0^\alpha}^{\theta}}=\frac{\left(M\kappa(\kappa+1)\right)^{\frac{\theta}{\theta^*}}}{(1-\delta)^\theta}.$$

Note that,
$$I_{22}\leq \left(2\|D\Phi_\eps-ED\Phi_0^\eps\|_{L^\infty(\mathbb{R}^m,\, \mathcal{L}(\mathbb{R}^m, X_\eps^\alpha))}\right)^{1-\frac{\theta}{\theta^*}}.$$

Hence, for $\theta<\theta^*$,
$$\|\Phi_\eps- E\Phi_0^\eps\|_{C^{1, \theta}(\mathbb{R}^m, X_\eps^\alpha)}\leq $$
$$\leq \|\Phi_\eps- E\Phi_0^\eps\|_{L^\infty(\mathbb{R}^m, X_\eps^\alpha)}+\|D\Phi_\eps-ED\Phi_0^\eps\|_{L^\infty(\mathbb{R}^m,\, \mathcal{L}(\mathbb{R}^m, X_\eps^\alpha))}+$$
$$+\frac{\left(M\kappa(\kappa+1)\right)^{\frac{\theta}{\theta^*}}}{(1-\delta)^\theta} \left(2\|D\Phi_\eps-ED\Phi_0^\eps\|_{L^\infty(\mathbb{R}^m,\, \mathcal{L}(\mathbb{R}^m, X_\eps^\alpha))}\right)^{1-\frac{\theta}{\theta^*}}.$$
By Theorem \ref{distaciavariedadesinerciales} and Proposition \ref{differential-convergence}, we have 

$$\|\Phi_\eps- E\Phi_0^\eps\|_{C^{1, \theta}(\mathbb{R}^m, X_\eps^\alpha)}\leq $$
$$\leq C[\tau(\eps)|\log(\tau(\eps))|+\rho(\eps)]+2 C\left[\beta(\varepsilon)+(\tau(\varepsilon)|\log(\tau(\varepsilon))|+\rho(\varepsilon))^{\theta^*}\right]+$$
$$+\frac{\left(M\kappa(\kappa+1)\right)^{\frac{\theta}{\theta^*}}}{(1-\delta)^\theta} \left(4 C\left[\beta(\varepsilon)+(\tau(\varepsilon)|\log(\tau(\varepsilon))|+\rho(\varepsilon))^{\theta^*}\right]\right)^{1-\frac{\theta}{\theta^*}}\leq$$
$$\leq \mathbf{C} \left(\left[\beta(\varepsilon)+(\tau(\varepsilon)|\log(\tau(\varepsilon))|+\rho(\varepsilon))^{\theta^*}\right]\right)^{1-\frac{\theta}{\theta^*}},$$
which shows the result.

\end{proof}
\par\bigskip\bigskip



\begin{thebibliography}{99}


\bibitem{Arrieta-Santamaria-DCDS}  J.M. Arrieta, E. Santamar\'ia, {\sl Estimates on the distance of Inertial Manifolds},  Discrete and Continuous Dynamical Systems A, 34, Vol 10 pp. 3921-3944 (2014)


\bibitem{Arrieta-Santamaria-2}  J.M. Arrieta, E. Santamar\'ia, {\sl Distance of attractors for thin domains}, (In preparation)
\bibitem{B&V2}A. V. Babin and M. I. Vishik, {\sl Attractors of Evolution Equations},  Studies in Mathematics and its Applications, 25. North-Holland Publishing Co., Amsterdam,  (1992).
\bibitem{Bates-Lu-Zeng1998} Bates, P.W.; Lu, K.; Zeng, C. {\sl Existence and Persistence of Invariant Manifolds for Semiflows
in Banach Space}  Mem. Am. Math. Soc. bf 135, (1998), no. 645.
\bibitem{LibroAlexandre}A. N. Carvalho, J. Langa, J. C. Robinson, {\sl Attractors for Infinite-Dimensional Non-Autonomous Dynamical-Systems}, Applied Mathematical Sciences, Vol. 182, Springer, (2012).
\bibitem{ChaoBiaoCo}Shui-Nee Chow, Xiao-Biao Lin and Kening Lu, {\sl Smooth Invariant Foliations in Infinite Dimensional Spaces},  Journal of Differential Equations 94 (1991), no. 2, 266–291
\bibitem{ChowLuSell}S. Chow, K. Lu and G. R. Sell, {\sl Smoothness of Inertial Manifolds}, Journal of Mathematical Analysis and Applications, 169, no. 1, 283-312, (1992).
\bibitem{Cholewa}J. W. Cholewa and T. Dlotko, {\sl Global Attractors in Abstract Parabolic Problems},  London Mathematical Society Lecture Note Series, 278. Cambridge University Press, Cambridge, (2000)
\bibitem{Hale}Jack K. Hale, {\sl Asymptotic Behavior of Dissipative Systems}, American Mathematical Society (1988).
\bibitem{Hale&Raugel3}Jack K. Hale and Genevieve Raugel, {\sl Reaction-Diffusion Equation on Thin Domains}, J. Math. Pures et Appl. (9) 71 (1992), no. 1, 33-95.
\bibitem{Henry1}Daniel B. Henry, {\sl Geometric Theory of Semilinear Parabolic Equations}, Lecture Notes in Mathematics, 840. Springer-Verlag, Berlin-New York, (1981).
\bibitem{Jones}Don A. Jones, Andrew M. Stuart and Edriss S. Titi, {\sl Persistence of Invariant Sets for Dissipative Evolution Equations}, Journal of Mathematical Analysis and Applications  219, 479-502 (1998)
\bibitem{Ng2013} P. S. Ngiamsunthorn, {\sl Invariant manifolds for parabolic equations under perturbation of the domain}, Nonlinear Analysis TMA 80, pp 23-48, (2013)



\bibitem{Raugel}Genevieve Raugel, {\sl  Dynamics of partial differential equations on thin domains}. Dynamical systems (Montecatini Terme, 1994), 208-315, Lecture Notes in Math., 1609, Springer, Berlin, (1995).
\bibitem{JamesRobinson}James C. Robinson, {\sl Infinite-dimensional dynamical systems. An introduction to dissipative parabolic PDEs and the theory of global attractors},  Cambridge Texts in Applied Mathematics. Cambridge University Press, Cambridge, 2001
\bibitem{Sell&You}George R. Sell and Yuncheng You, {\sl Dynamics of Evolutionary Equations}, Applied Mathematical Sciences, 143, Springer (2002). 
\bibitem{Varchon2012} N. Varchon,  {\sl Domain perturbation and invariant manifolds}, J. Evol. Equ. 12 (2012), 547-569



\end{thebibliography}
\end{document}